\documentclass{amsart}
\usepackage[latin1]{inputenc}
\usepackage{amssymb,amsmath,amsthm}
\usepackage{amsfonts}
\usepackage{paralist}

\newcommand{\cl}[1]{\ensuremath{\mathcal #1}}
\newcommand{\vect}[1]{\ensuremath{\mathbf{#1}}}
\newcommand{\mtrx}[1]{\ensuremath{\mathbf{#1}}}
\newcommand{\card}[1]{\ensuremath{\lvert{#1}\rvert}}
\DeclareMathOperator{\Pol}{Pol}
\DeclareMathOperator{\Inv}{Inv}

\theoremstyle{plain}
\newtheorem{theorem}{Theorem}[section]

\newtheorem{lemma}[theorem]{Lemma}

\theoremstyle{definition}

\newtheorem{clm}{Claim}
\newtheorem{example}{Example}

\theoremstyle{remark}
\newtheorem{remark}{Remark}
\newtheorem{fact}{Fact}

\begin{document}
\title{Closed classes of functions, generalized constraints and clusters}
\author{Erkko Lehtonen}
\address{Department of Mathematics \\
Tampere University of Technology \\
P.O. Box 553 \\
FI--33101 Tampere \\
Finland}
\address{Department of Combinatorics and Optimization \\
University of Waterloo \\
200 University Avenue West \\
Waterloo, Ontario, N2L 3G1 \\
Canada}
\email{erkko.lehtonen@tut.fi}
\thanks{This research was supported by the Academy of Finland, grant \#{}120307.}
\date{\today}
\begin{abstract}
Classes of functions of several variables on arbitrary non-empty domains that are closed under permutation of variables and addition of dummy variables are characterized in terms of generalized constraints, and hereby Hellerstein's Galois theory of functions and generalized constraints is extended to infinite domains. Furthermore, classes of operations on arbitrary non-empty domains that are closed under permutation of variables, addition of dummy variables and composition are characterized in terms of clusters, and a Galois connection is established between operations and clusters.
\end{abstract}

\maketitle

\section{Introduction}

Iterative algebras, as introduced by Mal'cev \cite{Malcev}, are classes of operations on a fixed base set that are closed under permutation of variables, addition of dummy variables, identification of variables, and composition. Clones are iterative algebras containing all projections, and they can be characterized as classes of operations that preserve sets of relations. The Galois connection between operations and relations given by the ``preservation'' relation is known as the $\Pol$--$\Inv$ theory. The closed classes of operations and the closed classes of relations on finite domains were first described, by explicit closure conditions, by Geiger \cite{Geiger} and independently by Bodnarchuk, Kaluzhnin, Kotov and Romov \cite{BKKR}. These results were extended to infinite domains by Szabó \cite{Szabo} and independently by Pöschel \cite{Poschel}. For general background on function and relation algebras, see the monographs by Pöschel and Kaluzhnin \cite{PK} and Lau \cite{Lau}. For additional information on clones, see the monograph by Szendrei \cite{Szendrei}.

Classes of functions that are closed under only some of the iterative algebra operations have been studied by several authors, and analogous Galois theories have been developed for these variant notions of closure. While for these variants the primal objects are still functions, the dual objects are not relations but something more general. Pippenger \cite{Pippenger} introduced the notion of a constraint and showed that the classes of finite functions that are closed under \emph{identification minors} (permutation of variables, identification of variables, and addition of dummy variables) are characterized by constraints, i.e., pairs of relations. He also gave closure conditions for classes of constraints that are characterized by functions. These results were extended to functions and constraints on arbitrary, possibly infinite domains by Couceiro and Foldes \cite{CF}.

Hellerstein \cite{Hellerstein} generalized the notion of a constraint and showed that the classes of finite functions that are closed under \emph{special minors} (permutation of variables and addition of dummy variables) are characterized by generalized constraints. She also gave closure conditions for classes of generalized constraints on finite sets that are characterized by functions. The first objective of the current paper is to extend Hellerstein's Galois theory of functions and generalized constraints to arbitrary, possibly infinite domains.

The second objective of this paper is to describe the classes of operations on arbitrary non-empty sets that are closed under the iterative algebra operations except for identification of variables, i.e., closed under permutation of variables, addition of dummy variables, and composition. We show that the classes that contain all projections and are closed under the operations mentioned are characterized by so-called clusters, which are defined as downward closed sets of multisets of $m$-tuples on the base set. We also give closure conditions for classes of clusters that are characterized by functions and thus establish a Galois theory of operations and clusters on arbitrary, possibly infinite domains.

\section{Preliminaries}
\label{sec:preliminaries}

We denote the set of nonnegative integers by $\mathbb{N}$, and we regards its elements as ordinals, i.e., $n \in \mathbb{N}$ is the set of lesser ordinals $\{0, 1, \dotsc, n - 1\}$. Thus, an $n$-tuple $\vect{a} \in A^n$ is formally a map $\vect{a} \colon \{0, 1, \dotsc, n - 1\} \to A$. The notation $(a_i \mid i \in n)$ means the $n$-tuple mapping $i$ to $a_i$ for each $i \in n$. The notation $(a_1, \dotsc, a_n)$ means the $n$-tuple mapping $i$ to $a_{i+1}$ for each $i \in n$.

Let $A$ and $B$ be arbitrary nonempty sets. A \emph{function of several variables from $A$ to $B$} is a map $f \colon A^n \to B$ for some integer $n \geq 1$, called the \emph{arity} of $f$. Subsets $\cl{C} \subseteq \bigcup_{n \geq 1} B^{A^n}$ are called \emph{classes.} For a class $\cl{C}$, we denote by $\cl{C}^{(n)}$ the \emph{$n$-ary part} of $\cl{C}$, i.e., the set of $n$-ary functions in $\cl{C}$. In the case that $A = B$, we call maps $f \colon A^n \to A$ \emph{operations on $A$.} The set of all operations on $A$ is denoted by $\cl{O}_A = \bigcup_{n \geq 1} A^{A^n}$. 

For maps $f \colon A \to B$ and $g \colon C \to D$, the composition $g \circ f$ is defined only if the codomain $B$ of $f$ coincides with the domain $C$ of $g$. Removing this restriction, the \emph{concatenation} of $f$ and $g$ is defined to be the map $gf \colon f^{-1}[B \cap C] \to D$ given by the rule $(gf)(a) = g \bigl( f(a) \bigr)$ for all $a \in f^{-1}[B \cap C]$. Clearly, if $B = C$, then $gf = g \circ f$; thus functional composition is subsumed and extended by concatenation. Concatenation is associative, i.e., for any maps $f$, $g$, $h$, we have $h(gf) = (hg)f$.

For a family $(g_i)_{i \in I}$ of maps $g_i \colon A_i \to B_i$ such that $A_i \cap A_j = \emptyset$ whenever $i \neq j$, we define the (\emph{piecewise}) \emph{sum of the family $(g_i)_{i \in I}$} to be the map $\sum_{i \in I} g_i \colon \bigcup_{i \in I} A_i \to \bigcup_{i \in I} B_i$ whose restriction to each $A_i$ coincides with $g_i$. If $I$ is a two-element set, say $I = \{1, 2\}$, then we write $g_1 + g_2$. Clearly, this operation is associative and commutative.

Concatenation is distributive over summation, i.e., for any family $(g_i)_{i \in I}$ of maps on disjoint domains and any map $f$,
\[
\bigl( \sum_{i \in I} g_i \bigr) f = \sum_{i \in I} (g_i f)
\qquad \text{and} \qquad
f \bigl( \sum_{i \in I} g_i \bigr) = \sum_{i \in I} (f g_i).
\]
In particular, if $g_1$ and $g_2$ are maps with disjoint domains, then
\[
(g_1 + g_2) f = (g_1 f) + (g_2 f)
\qquad \text{and} \qquad
f (g_1 + g_2) = (f g_1) + (f g_2).
\]

Let $g_1, \dotsc, g_n$ be maps from $A$ to $B$. The $n$-tuple $(g_1, \dotsc, g_n)$ determines a \emph{vector-valued map} $g \colon A \to B^n$, given by $g(a) = \bigl( g_1(a), \dotsc, g_n(a) \bigr)$ for every $a \in A$. For $f \colon B^n \to C$, the composition $f \circ g$ is a map from $A$ to $C$, denoted by $f(g_1, \dotsc, g_n)$, and called the \emph{composition of $f$ with $g_1, \dotsc, g_n$.} Suppose that $A \cap A' = \emptyset$ and $g'_1, \dotsc, g'_n$ are maps from $A'$ to $B$. Let $g$ and $g'$ be the vector-valued maps determined by $(g_1, \dotsc, g_n)$ and $(g'_1, \dotsc, g'_n)$, respectively. We have that $f(g + g') = (fg) + (fg')$, i.e.,
\[
f \bigl( (g_1 + g'_1), \dotsc, (g_n + g'_n) \bigr) = f(g_1, \dotsc, g_n) + f(g'_1, \dotsc, g'_n).
\]

For $B \subseteq A$, $\iota_{A B}$ denotes the canonical injection (inclusion map) from $B$ to $A$. Thus the \emph{restriction} $f|_B$ of any map $f \colon A \to C$ to the subset $B$ is given by $f|_B = f \iota_{A B}$.

Mal'cev \cite{Malcev} introduced the operations $\zeta$, $\tau$, $\Delta$, $\nabla$, $\ast$ on the set $\cl{O}_A$ of all operations on $A$, defined as follows for arbitrary $f \in \cl{O}_A^{(n)}$, $g \in \cl{O}_A^{(m)}$:
\begin{gather*}
(\zeta f)(x_1, x_2, \dotsc, x_n) = f(x_2, x_3, \dotsc, x_n, x_1), \\
(\tau f)(x_1, x_2, \dotsc, x_n) = f(x_2, x_1, x_3, \dotsc, x_n), \\
(\Delta f)(x_1, x_2, \dotsc, x_{n-1}) = f(x_1, x_1, x_2, \dotsc, x_{n-1})
\end{gather*}
for $n > 1$, $\zeta f = \tau f = \Delta f = f$ for $n = 1$, and
\begin{gather*}
(\nabla f)(x_1, x_2, \dotsc, x_{n+1}) = f(x_2, \dotsc, x_{n+1}), \\
(f \ast g)(x_1, x_2, \dotsc, x_{m+n-1}) = f \bigl( g(x_1, x_2, \dotsc, x_m), x_{m+1}, \dotsc, x_{m+n-1} \bigr).
\end{gather*}
The operations $\zeta$ and $\tau$ are collectively referred to as \emph{permutation of variables,} $\Delta$ is called \emph{identification of variables} (or \emph{diagonalization}), $\nabla$ is called \emph{addition of a dummy variable} (or \emph{cylindrification}), and $\ast$ is called \emph{composition.} The subalgebras of $(\cl{O}_A; \zeta, \tau, \Delta, \nabla, \ast)$ are called \emph{iterative algebras.} Clones are iterative algebras that contain all \emph{projections} $(x_1, \dotsc, x_n) \mapsto x_i$, $1 \leq i \leq n$. Note that the operations $\zeta$, $\tau$, $\Delta$ and $\nabla$ can be defined in an analogous way on the set of all functions of several variables from $A$ to $B$, and we will subsequently make reference to these operations in this more general setting as well.

A \emph{finite multiset} $S$ on a set $A$ is a map $\nu_S \colon A \to \mathbb{N}$, called a \emph{multiplicity function,} such that $\sum_{x \in A} \nu_S(x) < \infty$. The integer $\sum_{x \in A} \nu_S(x)$ is called the \emph{cardinality} of $S$, and it is denoted by $\card{S}$. The number $\nu_S(x)$ is called the \emph{multiplicity} of $x$ in $S$. We may represent a finite multiset $S$ by giving a list enclosed in set brackets, i.e., $\{a_1, \dotsc, a_n\}$, where each element $x \in A$ occurs as many times as the multiplicity $\nu_S(x)$ indicates. If $S'$ is another multiset on $A$ corresponding to $\nu_{S'} \colon A \to \mathbb{N}$, then we say that $S'$ is a \emph{submultiset} of $S$, denoted $S' \subseteq S$, if $\nu_{S'}(x) \leq \nu_S(x)$ for all $x \in A$. We denote the set of all finite multisets on $A$ by $\mathcal{M}(A)$. The set $\mathcal{M}(A)$ is partially ordered by the multiset inclusion relation ``$\subseteq$''. The \emph{join} $S \uplus S'$ and the \emph{difference} $S \setminus S'$ of multisets $S$ and $S'$ are determined by the multiplicity functions $\nu_{S \uplus S'}(x) = \nu_S(x) + \nu_{S'}(x)$ and $\nu_{S \setminus S'}(x) = \max \{\nu_S(x) - \nu_{S'}(x), 0\}$, respectively. The \emph{empty multiset} on $A$ is the zero function, and it is denoted by $\varepsilon$. A \emph{partition} of a finite multiset $S$ on $A$ is a multiset $\{S_1, \dotsc, S_n\}$ (on the set of all finite multisets on $A$) of non-empty finite multisets on $A$ such that $S = S_1 \uplus \dotsb \uplus S_n$.

We view an $m \times n$ matrix $\mtrx{M} \in A^{m \times n}$ with entries in $A$ as an $n$-tuple of $m$-tuples $\mtrx{M} = (\vect{a}^1, \dotsc, \vect{a}^n)$. The $m$-tuples $\vect{a}^1, \dotsc, \vect{a}^n$ are called the \emph{columns} of $\mtrx{M}$. For $i \in m$, the $n$-tuple $\bigl( \vect{a}^1(i), \dotsc, \vect{a}^n(i) \bigr)$ is called \emph{row} $i$ of $\mtrx{M}$. If for $1 \leq i \leq p$, $\mtrx{M}_i = (\vect{a}^i_1, \dotsc, \vect{a}^i_{n_i})$ is an $m \times n_i$ matrix, then we denote by $[\mtrx{M}_1 | \mtrx{M}_2 | \dotsb | \mtrx{M}_p]$ the $m \times \sum_{i = 1}^p n_i$ matrix $(\vect{a}^1_1, \dotsc, \vect{a}^1_{n_1}, \vect{a}^2_1, \dotsc, \vect{a}^2_{n_2}, \dotsc, \vect{a}^p_1, \dotsc, \vect{a}^p_{n_p})$. An \emph{empty matrix} has no columns and is denoted by $()$.

For an $m \times n$ matrix $\mtrx{M} \in A^{m \times n}$, the \emph{multiset of columns} of $\mtrx{M}$ is the multiset $\mtrx{M}^*$ on $A^m$ defined by the multiplicity function $\chi_\mtrx{M} \colon A^m \to \mathbb{N}$, called the \emph{characteristic function} of $\mtrx{M}$, which maps each $m$-tuple $\vect{a} \in A^m$ to the number of times $\vect{a}$ occurs as a column of $\mtrx{M}$. A matrix $\mtrx{N} \in A^{m \times n'}$ is a \emph{submatrix} of $\mtrx{M} \in A^{m \times n}$ if $\mtrx{N}^* \subseteq \mtrx{M}^*$, i.e., $\chi_\mtrx{N}(\vect{a}) \leq \chi_\mtrx{M}(\vect{a})$ for all $\vect{a} \in A^m$.

For a function $f \colon A^n \to B$ and a matrix $\mtrx{M} = (\vect{a}^1, \dotsc, \vect{a}^n) \in A^{m \times n}$, we denote by $f \mtrx{M}$ the $m$-tuple $\Bigl( f \bigl( \vect{a}^1(i), \dotsc, \vect{a}^n(i) \bigr) \Bigm| i \in m \Bigr)$ in $B^m$. Observe that this notation is in accordance with the notation for concatenation of mappings. Since a matrix $\mtrx{M} = (\vect{a}^1, \dotsc, \vect{a}^n)$ is an $n$-tuple of $m$-tuples $\vect{a}^i \colon m \to A$, $1 \leq i \leq n$, the composition of the vector-valued map $(\vect{a}^1, \dotsc, \vect{a}^n) \colon m \to A^n$ with $f \colon A^n \to B$ gives rise to the $m$-tuple $f(\vect{a}^1, \dotsc, \vect{a}^n) \colon m \to B$.

\section{Classes of functions closed under permutation of variables and addition of dummy variables}
\label{sec:functions}

Hellerstein \cite{Hellerstein} showed that the classes of functions of several variables on finite domains that are closed under permutation of variables and addition of dummy variables are characterized by generalized constraints. We extend Hellerstein's Galois theory of functions and generalized constraints to arbitrary, possibly infinite domains. Our results and proofs closely follow Hellerstein's analogous statements for functions and generalized constraints on finite domains, which in turn are adaptations of those by Pippenger \cite{Pippenger}, Geiger \cite{Geiger} and Bodnarchuk \textit{et al.}\ \cite{BKKR}.

For $m \geq 1$, an $m$-ary \emph{repetition function} on $A$ is a map $\phi \colon A^m \to \mathbb{N} \cup \{\infty\}$. An $m$-ary \emph{generalized constraint} from $A$ to $B$ is a pair $(\phi, S)$ where $\phi$ is an $m$-ary repetition function on $A$ called the \emph{antecedent}, and $S \subseteq B^m$ is called the \emph{consequent}. The number $m$ is called the \emph{arity} of the generalized constraint.

If $\mtrx{M} \in A^{m \times n}$ is an $m \times n$ matrix with entries from $A$ and $\phi$ is an $m$-ary repetition function on $A$, we write $\mtrx{M} \prec \phi$ to mean that each $m$-tuple $\vect{a} \in A^m$ occurs as a column of $\mtrx{M}$ at most $\phi(\vect{a})$ times. If $f \colon A^n \to B$ and $(\phi, S)$ is an $m$-ary generalized constraint from $A$ to $B$, we say that $f$ \emph{satisfies} $(\phi, S)$ if for every $m \times n$ matrix $\mtrx{M}$ such that $\mtrx{M} \prec \phi$, we have that $f \mtrx{M} \in S$.

We say that a class $\cl{C}$ of functions of several variables from $A$ to $B$ is \emph{characterized} by a set $\cl{T}$ of generalized constraints from $A$ to $B$ if $\cl{C}$ is precisely the set of functions satisfying all generalized constraints in $\cl{T}$. Similarly, $\cl{T}$ is said to be \emph{characterized} by $\cl{C}$ if $\cl{T}$ is precisely the set of generalized constraints satisfied by all functions in $\cl{C}$.

A class $\cl{C}$ of functions of several variables from $A$ to $B$ is \emph{locally closed} if for every $g \colon A^n \to B$, it holds that $g \in \cl{C}$ whenever for every finite $F \subseteq A^n$ there is a $f \in \cl{C}^{(n)}$ such that $g|_F = f|_F$.

Recall that $\chi_\mtrx{M}$ denotes the characteristic function of a matrix $\mtrx{M}$. For any class $\cl{C}$ of functions of several variables from $A$ to $B$ and for an $m \times n$ matrix $\mtrx{M} \in A^{m \times n}$, we denote $\cl{C} \mtrx{M} = \{ f \mtrx{M} : f \in \cl{C}^{(n)}\}$.

\begin{lemma}
\label{lem:chiM}
If $\cl{C}$ is a class of functions of several variables from $A$ to $B$ that is closed under permutation of variables and addition of dummy variables, then for every matrix $\mtrx{M} \in A^{m \times n}$, the generalized constraint $(\chi_\mtrx{M}, \cl{C} \mtrx{M})$ is satisfied by all functions in $\cl{C}$.
\end{lemma}
\begin{proof}
Let $f' \in \cl{C}$ be $n'$-ary, and let $\mtrx{M}'$ be an $m \times n'$ matrix such that $\mtrx{M}' \prec \chi_\mtrx{M}$. Then there exists an injection $\sigma \colon \{1, \dotsc, n'\} \to \{1, \dotsc, n\}$ such that column $i$ of $\mtrx{M}'$ equals column $\sigma(i)$ of $\mtrx{M}$. Let $f$ be the $n$-ary function defined by $f(x_1, \dotsc, x_n) = f'(x_{\sigma(1)}, \dotsc, x_{\sigma(n')})$. We have that $f' \mtrx{M}' = f \mtrx{M}$. Since $f$ is obtained from $f'$ by permutation of variables and addition of dummy variables, it is a member of $\cl{C}$, and hence $f \mtrx{M} \in \cl{C} \mtrx{M}$. Thus $f' \mtrx{M}' \in \cl{C} \mtrx{M}$ and so $f'$ satisfies $(\chi_\mtrx{M}, \cl{C} \mtrx{M})$.
\end{proof}

We are now ready to describe the classes of functions that are characterized by generalized constraints.

\begin{theorem}
\label{ExtHellersteinTheorem5.1}
Let $A$ and $B$ be arbitrary, possibly infinite non-empty sets. For any class $\cl{F}$ of functions of several variables from $A$ to $B$, the following two conditions are equivalent:
\begin{compactenum}[(i)]
\item $\cl{F}$ is locally closed and it is closed under permutation of variables and addition of dummy variables.
\item $\cl{F}$ can be characterized by a set $\cl{T}$ of generalized constraints from $A$ to $B$.
\end{compactenum}
\end{theorem}
\begin{proof}
$\text{(ii)} \Rightarrow \text{(i)}$:
As observed in the case of finite domains by Hellerstein \cite{Hellerstein}, it is easy to see, also in general, that the set of functions satisfying a generalized constraint $(\phi, S)$ is closed under permutation of variables and addition of dummy variables. Thus any class of functions characterized by a set $\cl{T}$ of generalized constraints is closed under the operations considered.

To show that $\cl{F}$ is locally closed, consider $f \notin \cl{F}$, and let $(\phi, S) \in \cl{T}$ be a generalized constraint that is not satisfied by $f$ (but it is satisfied by every function $g \in \cl{F}$). Assume that $f$ is $n$-ary. Then, for some matrix $\mtrx{M} \prec \phi$, $f \mtrx{M} \notin S$, but $g \mtrx{M} \in S$ for every $g \in \cl{F}^{(n)}$. Thus, the restriction of $f$ to the finite set of rows of $M$ does not coincide with that of any member of $\cl{F}$.

$\text{(i)} \Rightarrow \text{(ii)}$:
We need to show that for every function $g \notin \cl{F}$, there exists a generalized constraint that is satisfied by every function in $\cl{F}$ but not by $g$. The set of all such ``separating'' generalized constraints, for each $g \notin \cl{F}$, characterizes $\cl{F}$.

The case $\cl{F} = \emptyset$ being trivial, we assume that $\cl{F} \neq \emptyset$. Suppose that $g \notin \cl{F}$ is $n$-ary. Since $\cl{F}$ is locally closed, there is a finite subset $F \subseteq A^n$ such that $g|_F \neq f|_F$ for every $f \in \cl{F}^{(n)}$. Clearly $F$ is non-empty. Let $\mtrx{M}$ be a $\card{F} \times n$ matrix whose rows are the elements of $F$ in some fixed order, and consider the generalized constraint $(\chi_\mtrx{M}, \cl{F} \mtrx{M})$. By Lemma \ref{lem:chiM}, every function in $\cl{F}$ satisfies $(\chi_\mtrx{M}, \cl{F} \mtrx{M})$. But $g \mtrx{M} = g|_F \mtrx{M} \notin \cl{F} \mtrx{M}$, and hence $g$ does not satisfy $(\chi_\mtrx{M}, \cl{F} \mtrx{M})$ either. This completes the proof of the theorem.
\end{proof}

\section{Closure conditions for generalized constraints}

In order to describe the sets of generalized constraints on arbitrary, possibly infinite domains that are characterized by functions, we need a more general closure condition than the one given by Hellerstein \cite{Hellerstein} for finite domains. We will follow Couceiro and Foldes's \cite{CF} proof techniques and adapt their notion of conjunctive minor to generalized constraints.

Let $(\phi, S)$ and $(\phi', S')$ be $m$-ary generalized constraints from $A$ to $B$. If $\phi' \leq \phi$ and $S' = S$ then we say that $(\phi', S')$ is obtained from $(\phi, S)$ by \emph{restricting the antecedent.} If $\phi' = \phi$ and $S' \supseteq S$, then we say that $(\phi', S')$ is obtained from $(\phi, S)$ by \emph{extending the consequent.} If $(\phi', S')$ is obtained from $(\phi, S)$ be restricting the antecedent or extending the consequent, or by the combination of the two (i.e., $\phi' \leq \phi$ and $S' \supseteq S$), then we say that $(\phi', S')$ is a \emph{relaxation} of $(\phi, S)$. If $(\phi, S_j)_{j \in J}$ is a non-empty family of generalized constraints with the same antecedent, then we say that the generalized constraint $(\phi, \bigcap_{j \in J} S_j)$ is obtained from $(\phi, S_j)_{j \in J}$ by \emph{intersecting consequents.}

Let $(\phi, S)$ be an $m$-ary generalized constraint from $A$ to $B$, and let $\phi'$ be an $m$-ary repetition function on $A$. If $\phi' \leq \phi$ and the set $\{\vect{a} \in A^m : \phi'(\vect{a}) \neq 0\}$ is finite, then we say that the generalized constraint $(\phi', S)$ is obtained from $(\phi, S)$ by a \emph{finite restriction of the antecedent.} We say that a set $\cl{T}$ of generalized constraints from $A$ to $B$ is \emph{locally closed} if for every generalized constraint $(\phi, S)$, it holds that $(\phi, S) \in \cl{T}$ whenever every generalized constraint obtained from $(\phi, S)$ by a finite restriction of the antecedent belongs to $\cl{T}$.

Let $m$ and $n$ be positive integers (viewed as ordinals, i.e., $m = \{0, \dotsc, m-1\}$). Let $h \colon n \to m \cup V$ where $V$ is an arbitrary set of symbols disjoint from the ordinals, called \emph{existentially quantified indeterminate indices,} or simply \emph{indeterminates,} and let $\sigma \colon V \to A$ be any map, called a \emph{Skolem map.} Then each $m$-tuple $\vect{a} \in A^m$, being a map $\vect{a} \colon m \to A$, gives rise to an $n$-tuple $(\vect{a} + \sigma) h \in A^n$.

Let $H = (h_j)_{j \in J}$ be a non-empty family of maps $h_j \colon n_j \to m \cup V$, where each $n_j$ is a positive integer. Then $H$ is called a \emph{minor formation scheme} with \emph{target} $m$, \emph{indeterminate set} $V$, and \emph{source family} $(n_j)_{j \in J}$. Let $(R_j)_{j \in J}$ be a family of relations (of various arities) on the same set $A$, each $R_j$ of arity $n_j$, and let $R$ be an $m$-ary relation on $A$. We say that $R$ is a \emph{restrictive conjunctive minor} of the family $(R_j)_{j \in J}$ \emph{via} $H$, or simply a \emph{restrictive conjunctive minor} of the family $(R_j)_{j \in J}$ if, for every $m$-tuple $\vect{a} \in A^m$, the condition $\vect{a} \in R$ implies that there is a Skolem map $\sigma \colon V \to A$ such that, for all $j \in J$, we have $(\vect{a} + \sigma) h_j \in R_j$. On the other hand, if, for every $m$-tuple $\vect{a} \in A^m$, the condition $\vect{a} \in R$ holds whenever there is a Skolem map $\sigma \colon V \to A$ such that, for all $j \in J$, we have $(\vect{a} + \sigma) h_j \in R_j$, then we say that $R$ is an \emph{extensive conjunctive minor} of the family $(R_j)_{j \in J}$ \emph{via} $H$, or simply an \emph{extensive conjunctive minor} of the family $(R_j)_{j \in J}$. If $R$ is both a restrictive conjunctive minor and an extensive conjunctive minor of the family $(R_j)_{j \in J}$ via $H$, then $R$ is said to be a \emph{tight conjunctive minor} of the family $(R_j)_{j \in J}$ \emph{via} $H$, or simply a \emph{tight conjunctive minor} of the family. For a scheme $H$ and a family $(R_j)_{j \in J}$ of relations, there is a unique tight conjunctive minor of the family $(R_j)_{j \in J}$ via $H$.

We adapt these notions to repetition functions. Let $(\phi_j)_{j \in J}$ be a family of repetition functions (of various arities) on $A$, each $\phi_j$ of arity $n_j$, and let $\phi$ be an $m$-ary repetition function on $A$. We say that $\phi$ is a \emph{restrictive conjunctive minor} of the family $(\phi_j)_{j \in J}$ \emph{via} $H$, or simply a \emph{restrictive conjunctive minor} of the family $(\phi_j)_{j \in J}$ if, for every $m \times n$ matrix $\mtrx{M} = (\vect{a}^1, \dotsc, \vect{a}^n) \in A^{m \times n}$, the condition $\mtrx{M} \prec \phi$ implies that there are Skolem maps $\sigma_i \colon V \to A$, $1 \leq i \leq n$, such that, for all $j \in J$, we have $\bigl( (\vect{a}^1 + \sigma_1) h_j, \dotsc, (\vect{a}^n + \sigma_n) h_j \bigr) \prec \phi_j$. On the other hand, if, for every $m \times n$ matrix $\mtrx{M} = (\vect{a}^1, \dotsc, \vect{a}^n) \in A^{m \times n}$, the condition $\mtrx{M} \prec \phi$ holds whenever there are Skolem maps $\sigma_i \colon V \to A$, $1 \leq i \leq n$, such that, for all $j \in J$, we have $\bigl( (\vect{a}^1 + \sigma_1) h_j, \dotsc, (\vect{a}^n + \sigma_n) h_j \bigr) \prec \phi_j$, then we say that $\phi$ is an \emph{extensive conjunctive minor} of the family $(\phi_j)_{j \in J}$ \emph{via} $H$, or simply an \emph{extensive conjunctive minor} of the family $(\phi_j)_{j \in J}$. If $\phi$ is both a restrictive conjunctive minor and an extensive conjunctive minor of the family $(\phi_j)_{j \in J}$ via $H$, then $\phi$ is said to be a \emph{tight conjunctive minor} of the family $(\phi_j)_{j \in J}$ \emph{via} $H$, or simply a \emph{tight conjunctive minor} of the family.

\begin{remark}
If $\phi$ is a restrictive conjunctive minor of the family $(\phi_j)_{j \in J}$ of repetition functions via the scheme $(h_j)_{j \in J}$, then it holds for every $\vect{a} \in A^m$ that, for all $j \in J$,
\[
\sum_{\vect{b} \in \langle \vect{a} \rangle} \phi(\vect{b}) \leq \sum_{\vect{c} \in S^\vect{a}_j} \phi_j(\vect{c}),
\]
where $\langle \vect{a} \rangle = \{\vect{b} \in A^m : (\vect{b} + \sigma) h_j = (\vect{a} + \sigma) h_j\}$ for some Skolem map $\sigma \colon V \to A$, and $S^\vect{a}_j = \{(\vect{a} + \sigma) h_j \in A^{n_j} : \sigma \in A^V\}$. Note that the definition of $\langle \vect{a} \rangle$ does not depend on the choice of $\sigma$. Also, $S^\vect{a}_j = S^\vect{b}_j$ for every $\vect{b} \in \langle \vect{a} \rangle$.

Similarly, if $\phi$ is an extensive conjunctive minor of $(\phi_j)_{j \in J}$ via $(h_j)_{j \in J}$, then it holds for every $\vect{a} \in A^m$ that, for all $j \in J$,
\[
\sum_{\vect{b} \in \langle \vect{a} \rangle} \phi(\vect{b}) \geq \sum_{\vect{c} \in S^\vect{a}_j} \phi_j(\vect{c}).
\]
Consequently, for a tight conjunctive minor $\phi$ of $(\phi_j)_{j \in J}$ via $(h_j)_{j \in J}$, we have the equality
\[
\sum_{\vect{b} \in \langle \vect{a} \rangle} \phi(\vect{b}) = \sum_{\vect{c} \in S^\vect{a}_j} \phi_j(\vect{c}),
\]
but tight conjunctive minors of families of repetition functions are not unique.
\end{remark}

If $(\phi_j, S_j)_{j \in J}$ is a family of generalized constraints from $A$ to $B$ (of various arities) and $(\phi, S)$ is a generalized constraint from $A$ to $B$ such that for a scheme $H$, $\phi$ is a restrictive conjunctive minor of $(\phi_j)_{j \in J}$ via $H$ and $S$ is an extensive conjunctive minor of $(S_j)_{j \in J}$ via $H$, then $(\phi, S)$ is said to be a \emph{conjunctive minor} of the family $(\phi_j, S_j)_{j \in J}$ \emph{via} $H$, or simply a \emph{conjunctive minor} of the family of generalized constraints. If both $\phi$ and $S$ are tight conjunctive minors of the respective families via $H$, the generalized constraint $(\phi, S)$ is said to be a \emph{tight conjunctive minor} of the family $(\phi_j, S_j)_{j \in J}$ \emph{via} $H$, or simply a \emph{tight conjunctive minor} of the family of generalized constraints. Tight conjunctive minors of families of generalized constraints are not unique, but if both $(\phi, S)$ and $(\phi', S')$ are tight conjunctive minors of the same family of generalized constraints via the same scheme, then $S = S'$. If the minor formation scheme $H = (h_j)_{j \in J}$ and the family $(\phi_j, S_j)_{j \in J}$ are indexed by a singleton $J = \{0\}$, then a tight conjunctive minor $(\phi, S)$ of a family containing a single generalized constraint $(\phi_0, S_0)$ is called a \emph{simple minor} of $(\phi_0, S_0)$.

\begin{lemma}
\label{ExtCFLemma3.1}
Let $(\phi, S)$ be a conjunctive minor of a non-empty family $(\phi_j, S_j)_{j \in J}$ of generalized constraints from $A$ to $B$. If $f \colon A^n \to B$ satisfies every $(\phi_j, S_j)$, then $f$ satisfies $(\phi, S)$.
\end{lemma}
\begin{proof}
Let $(\phi, S)$ be an $m$-ary conjunctive minor of the family $(\phi_j, S_j)_{j \in J}$ via the scheme $H = (h_j)_{j \in J}$, $h_j \colon n_j \to m \cup V$. Let $\mtrx{M} = (\vect{a}^1, \dotsc, \vect{a}^n)$ be an $m \times n$ matrix such that $\mtrx{M} \prec \phi$. We need to prove that the $m$-tuple $f \mtrx{M}$ belongs to $S$. Since $\phi$ is a restrictive conjunctive minor of $(\phi_j)_{j \in J}$ via $H = (h_j)_{j \in J}$, there are Skolem maps $\sigma_i \colon V \to A$, $1 \leq i \leq n$, such that for every $j \in J$, for the matrix $\mtrx{M}_j = \bigl( (\vect{a}^1 + \sigma_1) h_j, \dotsc, (\vect{a}^n + \sigma_n) h_j \bigr)$, we have $\mtrx{M}_j \prec \phi_j$.

Since $S$ is an extensive conjunctive minor of $(S_j)_{j \in J}$ via the same scheme $H = (h_j)_{j \in J}$, to prove that $f \mtrx{M} \in S$, it suffices to give a Skolem map $\sigma \colon V \to B$ such that, for all $j \in J$, the $n_j$-tuple $(f \mtrx{M} + \sigma) h_j$ belongs to $S_j$. Let $\sigma = f(\sigma_1, \dotsc, \sigma_n)$. We have that, for each $j \in J$,
\begin{align*}
(f \mtrx{M} + \sigma) h_j
&= \bigl( f(\vect{a}^1, \dotsc, \vect{a}^n) + f(\sigma_1, \dotsc, \sigma_n) \bigr) h_j \\
&= \bigl( f(\vect{a}^1 + \sigma_1, \dotsc, \vect{a}^n + \sigma_n) \bigr) h_j \\
&= f \bigl( (\vect{a}^1 + \sigma_1) h_j, \dotsc, (\vect{a}^n + \sigma_n) h_j \bigr)
= f \mtrx{M}_j.
\end{align*}
Since $f$ is assumed to satisfy $(\phi_j, S_j)$, we have $f \mtrx{M}_j \in S_j$.
\end{proof}

We say that a set $\cl{T}$ of generalized constraints is \emph{closed under formation of conjunctive minors} if whenever every member of the non-empty family $(\phi_j, S_j)_{j \in J}$ of generalized constraints is in $\cl{T}$, all conjunctive minors of the family $(\phi_j, S_j)_{j \in J}$ are also in $\cl{T}$.

The formation of conjunctive minors subsumes the formation of simple minors as well as the operations of restricting the antecedent, extending the consequent, and intersecting consequents. Simple minors in turn subsume permutation of arguments, projection, identification of arguments, and addition of a dummy argument (see Hellerstein \cite{Hellerstein}).

An $m$-ary \emph{generalized equality constraint} is defined to be the generalized constraint $(\phi, S)$ such that $\phi(\vect{a}) = \infty$ if all components of $\vect{a} \in A^m$ are equal and $\phi(\vect{a}) = 0$ otherwise, and such that the elements of $S$ are exactly those $m$-tuples $\vect{b} \in B^m$ in which all components are equal. An $m$-ary \emph{generalized empty constraint} is defined to be the generalized constraint $(\phi, S)$ where $\phi(\vect{a}) = 0$ for all $\vect{a} \in A^m$ and $S = \emptyset$. An $m$-ary \emph{generalized trivial constraint} is defined to be the generalized constraint $(\phi, S)$ where $\phi(\vect{a}) = \infty$ for all $\vect{a} \in A^m$ and $S = B^m$.

\begin{lemma}
\label{lemma:all}
Let $\cl{T}$ be a set of generalized constraints that contains the binary generalized equality constraint and the unary generalized empty constraint. If $\cl{T}$ is closed under formation of conjunctive minors, then it contains all generalized trivial constraints, all generalized equality constraints, and all generalized empty constraints.
\end{lemma}
\begin{proof}
The unary generalized trivial constraint is a simple minor of the binary generalized equality constraint via the scheme $H = \{h\}$, where $h \colon 2 \to 1$ is given by $h(0) = h(1) = 0$ (by identification of arguments). The $m$-ary generalized trivial constraint is a simple minor of the unary generalized trivial constraint via the scheme $H = \{h\}$, where $h \colon 1 \to m$ is given by $h(0) = 0$ (by addition of $m-1$ dummy arguments).

For $m \geq 2$, the $m$-ary generalized equality constraint is a conjunctive minor of the binary generalized equality constraint via the scheme $H = (h_i)_{i \in m - 1}$, where $h_i \colon 2 \to m$ is given by $h_i(0) = i$, $h_i(1) = i + 1$ (by addition of $n-2$ dummy arguments, restricting antecedents and intersecting the consequents).

The $m$-ary generalized empty constraint is a simple minor of the unary generalized empty constraint via the scheme $H = \{h\}$, where $h \colon 1 \to m$ is given by $h(0) = 0$ (by addition of $m-1$ dummy arguments).
\end{proof}

Let $a_0, a_1, a_2, \dotsc$ be a sequence of natural numbers such that $a_i \leq a_{i + 1}$ for all $i \in \mathbb{N}$. If the sequence contains a maximum element, we define the limit of the sequence to be the value of that element. Otherwise we define the limit of the sequence to be $\infty$. This limit is denoted by $\lim_{i \to \infty} a_i$.

For any fixed domain $S$, we define a partial order $\leq$ on the set of all functions $\phi \colon S \to \mathbb{N} \cup \{\infty\}$ as follows: $\phi \leq \phi'$ if and only if for all $x \in S$, $\phi(x) \leq \phi'(x)$. Let $\phi_0, \phi_1, \phi_2, \dotsc$ be a sequence of functions $S \to \mathbb{N} \cup \{\infty\}$ such that $\phi_i \leq \phi_{i + 1}$ for all $i \in \mathbb{N}$. The limit of the sequence is defined to be the function $\phi \colon S \to \mathbb{N} \cup \{\infty\}$ such that for all $x \in S$, $\phi(x) = \lim_{i \to \infty} \phi_i(x)$. A subset $\cl{Q}$ of $(\mathbb{N} \cup \{\infty\})^S$ is \emph{chain complete} if $\cl{Q}$ contains the limits of all sequences of functions in $\cl{Q}$.

The following two lemmas are due to Hellerstein \cite[Lemmas 2.1, 2.2]{Hellerstein}.

\begin{lemma}
\label{HellersteinLemma2.1}
Let $S$ be a finite set. Let $\cl{Q}$ be a set of functions $\phi \colon S \to \mathbb{N} \cup \{\infty\}$. Then the number of maximal elements of $\cl{Q}$ is finite.
\end{lemma}
\begin{lemma}
\label{HellersteinLemma2.2}
Let $S$ be a finite set. Let $\cl{Q}$ be a set of functions $\phi \colon S \to \mathbb{N} \cup \{\infty\}$ such that $\cl{Q}$ is chain complete. Then for each element $\phi \in \cl{Q}$, there exists a maximal element $\phi'$ of $\cl{Q}$ such that $\phi \leq \phi'$.
\end{lemma}

We say that a set $\cl{T}$ of generalized constraints from $A$ to $B$ is closed under \emph{taking the limit of antecedents,} if whenever $(\phi_i, S)_{i \in \mathbb{N}}$ is a family of members of $\cl{T}$ such that $\phi_i \leq \phi_{i+1}$ for all $i \in \mathbb{N}$, $(\lim_{i \to \infty} \phi_i, S)$ is also a member of $\cl{T}$.

\begin{theorem}
\label{thm:gc}
Let $A$ and $B$ be arbitrary, possibly infinite non-empty sets. For any set $\cl{T}$ of generalized constraints from $A$ to $B$, the following two conditions are equivalent:
\begin{compactenum}[(i)]
\item $\cl{T}$ is locally closed and contains the binary generalized equality constraint and the unary generalized empty constraint, and it is closed under formation of conjunctive minors and taking the limit of antecedents.
\item $\cl{T}$ is characterized by some set of functions of several variables from $A$ to $B$.
\end{compactenum}
\end{theorem}
\begin{proof}
$\text{(ii)} \Rightarrow \text{(i)}$: It is clear that every function satisfies the generalized equality and empty constraints. It follows from Lemma \ref{ExtCFLemma3.1} that if a function satisfies every member of a non-empty family $(\phi_j, S_j)_{j \in J}$ of generalized constraints, then it satisfies every conjunctive minor of the family. It is also clear that if a function satisfies every member of a family $(\phi_i, S)_{i \in \mathbb{N}}$ of generalized constraints such that $\phi_i \leq \phi_{i+1}$ for all $i \in \mathbb{N}$, then it also satisfies $(\lim_{i \to \infty} \phi_i, S)$. Thus, we need to show that $\cl{T}$ is locally closed. For that, let $(\phi, S)$ be an $m$-ary generalized constraint not in $\cl{T}$. By (ii), there is an $n$-ary function $f$ that satisfies every generalized constraint in $\cl{T}$ but does not satisfy $(\phi, S)$. Thus, there is an $m \times n$ matrix $\mtrx{M} \prec \phi$ such that $f \mtrx{M} \notin S$. The constraint $(\chi_\mtrx{M}, S)$ is obtained from $(\phi, S)$ by a finite restriction of the antecedent, and $(\chi_\mtrx{M}, S) \notin \cl{T}$. This completes the proof of the implication $\text{(ii)} \Rightarrow \text{(i)}$.

\smallskip
$\text{(i)} \Rightarrow \text{(ii)}$: We need to extend the notions of relation and generalized constraint and allow them to have infinite arities. Functions remain finitary. These extended definitions have no bearing on the Theorem itself; they are only needed as a tool in its proof.

For any non-zero, possibly infinite ordinal $m$ (an ordinal $m$ is the set of lesser ordinals), an $m$-tuple is a map defined on $m$. The arities of relations and generalized constraints are thus allowed to be arbitrary non-zero, possibly infinite ordinals. In minor formation schemes, the target $m$ and the members $n_j$ of the source family are also allowed to be arbitrary non-zero, possibly infinite ordinals. For relations and repetition functions, we shall use the terms \emph{restrictive conjunctive $\infty$-minor} and \emph{extensive conjunctive $\infty$-minor} to indicate a restrictive or an extensive conjunctive minor via a scheme whose target and source ordinals may be infinite or finite. Similarly, for generalized constraints, we will use the terms \emph{conjunctive $\infty$-minor} and \emph{simple $\infty$-minor} to indicate conjunctive minors and simple minors via a scheme whose target and source ordinals may be infinite or finite. Thus in the sequel the use of the term ``minor'' without the prefix ``$\infty$'' continues to mean the respective minor via a scheme whose target and source ordinals are all finite. Matrices can also have infinitely many rows but only a finite number of columns; an $m \times n$ matrix $\mtrx{M}$, where $n$ is finite but $m$ may be finite or infinite, is an $n$-tuple of $m$-tuples $\mtrx{M} = (\vect{a}^1, \dotsc, \vect{a}^n)$.

In order to discuss the formation of repeated $\infty$-minors, we need the following definition. Let $H = (h_j)_{j \in J}$ be a minor formation scheme with target $m$, indeterminate set $V$ and source family $(n_j)_{j \in J}$, and, for each $j \in J$, let $H_j = (h^i_j)_{i \in I_j}$ be a scheme with target $n_j$, indeterminate set $V_j$ and source family $(n^i_j)_{i \in I_j}$. Assume that $V$ is disjoint from the $V_j$'s, and for distinct $j$'s the $V_j$'s are also pairwise disjoint. Then the \emph{composite scheme} $H(H_j : j \in J)$ is the scheme $K = (k^i_j)_{j \in J,\, i \in I_j}$ defined as follows:
\begin{compactenum}[(i)]
\item the target of $K$ is the target $m$ of $H$,
\item the source family of $K$ is $(n^i_j)_{j \in J,\, i \in I_j}$,
\item the indeterminate set of $K$ is $U = V \cup (\bigcup_{j \in J} V_j)$,
\item $k^i_j \colon n^i_j \to m \cup U$ is defined by $k^i_j = (h_j + \iota_{U V_j}) h^i_j$, where $\iota_{U V_j}$ is the canonical injection (inclusion map) from $V_j$ to $U$.
\end{compactenum}

\begin{clm}
\label{claim1}
If $(\phi, S)$ is a conjunctive $\infty$-minor of a non-empty family $(\phi_j, S_j)_{j \in J}$ of generalized constraints from $A$ to $B$ via the scheme $H$, and, for each $j \in J$, $(\phi_j, S_j)$ is a conjunctive $\infty$-minor of a non-empty family $(\phi^i_j, S^i_j)_{i \in I_j}$ via the scheme $H_j$, then $(\phi, S)$ is a conjunctive $\infty$-minor of the non-empty family $(\phi^i_j, S^i_j)_{j \in J,\, i \in I_j}$ via the composite scheme $K = H(H_j : j \in J)$.
\end{clm}

\noindent\textit{Proof of Claim \ref{claim1}.} First, we need to see that $\phi$ is a restrictive conjunctive $\infty$-minor of the family $(\phi^i_j)_{j \in J, i \in I_j}$ via $K$. Let $\mtrx{M} = (\vect{a}^1, \dotsc, \vect{a}^n)$ be an $m \times n$ matrix such that $\mtrx{M} \prec \phi$. This implies that there are Skolem maps $\sigma_i \colon V \to A$, $1 \leq i \leq n$, such that for all $j \in J$ we have $\bigl( (\vect{a}^1 + \sigma_1) h_j, \dotsc, (\vect{a}^n + \sigma_n) h_j) \bigr) \prec \phi_j$. This in turn implies that for all $j \in J$ there exist Skolem maps $\sigma^p_j \colon V_j \to A$, $1 \leq p \leq n$, such that for all $i \in I_j$ we have
\[
\Bigl( \bigl( (\vect{a}^1 + \sigma_1) h_j + \sigma^1_j \bigr) h^i_j, \dotsc, \bigl( (\vect{a}^n + \sigma_n) h_j + \sigma^n_j \bigr) h^i_j \Bigr) \prec \phi^i_j.
\]
Define the Skolem maps $\tau_p \colon U \to A$, $1 \leq p \leq n$, by $\tau_p = \sigma_p + \sum_{q \in J} \sigma^p_q$. Then for every $j \in J$ and $i \in I_j$, we have for $1 \leq p \leq n$,
\begin{equation}
\label{eq:Skolem}
\begin{split}
(\vect{a}^p + \tau_p) k^i_j
&= (\vect{a}^p + \sigma_p + \sum_{q \in J} \sigma^p_q)(h_j + \iota_{U V_j}) h^i_j \\
&= \Bigl( (\vect{a}^p + \sigma_p) h_j + \bigl( \sum_{q \in J} \sigma^p_q \bigr) h_j + (\vect{a}^p + \sigma_p) \iota_{U V_j} + \bigl( \sum_{q \in J} \sigma^p_q \bigr) \iota_{U V_j} \Bigr) h^i_j \\
&= \bigl( (\vect{a}^p + \sigma_p) h_j + \sigma^p_j \bigr) h^i_j,
\end{split}
\end{equation}
and hence
\[
\bigl( (\vect{a}^1 + \tau_1) k^i_j, \dotsc, (\vect{a}^n + \tau_n) k^i_j \bigr) \prec \phi^i_j.
\]

Second, we need to show that $S$ is an extensive conjunctive $\infty$-minor of the family $(S^i_j)_{j \in J,\, i \in I_j}$ via $K$. Let $\vect{b} \in B^m$ and assume that there is a Skolem map $\tau \colon U \to B$ such that for every $j \in J$ and $i \in I_j$, the $n^i_j$-tuple $(\vect{b} + \sigma) k^i_j$ is in $S^i_j$. We need to show that $\vect{b} \in S$. Define the Skolem maps $\sigma \colon V \to B$ and $\sigma_j \colon V_j \to B$ for every $j \in J$ such that each of these functions coincides with the restriction of $\tau$ to the respective domain, i.e., $\tau = \sigma + \sum_{j \in J} \sigma_j$. Similarly to \eqref{eq:Skolem},
\[
(\vect{b} + \tau) k^i_j = \bigl( (\vect{b} + \sigma) h_j + \sigma_j \bigr) h^i_j.
\]
Since $S_j$ is an extensive conjunctive $\infty$-minor of the family $(S^i_j)_{j \in J, i \in I_j}$ via the scheme $H_j$, we have $(\vect{b} + \sigma) h_j \in S_j$. Since the condition $(\vect{b} + \sigma) h_j \in S_j$ holds for all $j \in J$ and $S$ is an extensive conjunctive $\infty$-minor of the family $(S_j)_{j \in J}$ via $H$, we have that $\vect{b} \in S$. This completes the proof of Claim \ref{claim1}.

\smallskip
For a set $\cl{T}$ of generalized constraints from $A$ to $B$, we denote by $\cl{T}^\infty$ the set of those generalized constraints which are conjunctive $\infty$-minors of families of members of $\cl{T}$. This set $\cl{T}^\infty$ is the smallest set of generalized constraints containing $\cl{T}$ which is closed under formation of conjunctive $\infty$-minors, and it is called the \emph{conjunctive $\infty$-minor closure} of $\cl{T}$. In the sequel, we will make use of the following fact, which follows from Claim \ref{claim1}:
\begin{fact}
\label{fact:cmc}
Let $\cl{T}$ be a set of finitary generalized constraints from $A$ to $B$, and let $\cl{T}^\infty$ be its conjunctive $\infty$-minor closure. If $\cl{T}$ is closed under formation of conjunctive minors, then $\cl{T}$ is the set of all finitary generalized constraints belonging to $\cl{T}^\infty$.
\end{fact}

\begin{clm}
\label{claim2}
Let $\cl{T}$ be a locally closed set of finitary generalized constraints from $A$ to $B$ which contains the binary generalized equality constraint and the unary generalized empty constraint and is closed under formation of conjunctive minors and taking the limit of antecedents. Let $\cl{T}^\infty$ be the conjunctive $\infty$-minor closure of $\cl{T}$. Let $(\phi, S)$ be a finitary generalized constraint from $A$ to $B$ that is not in $\cl{T}$. Then there exists a function of several variables from $A$ to $B$ which satisfies every generalized constraint in $\cl{T}^\infty$ but does not satisfy $(\phi, S)$.
\end{clm}

\noindent\textit{Proof of Claim \ref{claim2}.} We shall construct a function $g$ which satisfies all generalized constraints in $\cl{T}^\infty$ but does not satisfy $(\phi, S)$.

Note that, by Fact \ref{fact:cmc}, $(\phi, S)$ cannot be in $\cl{T}^\infty$. Let $m$ be the arity of $(\phi, S)$. Since $\cl{T}$ is locally closed and $(\phi, S)$ does not belong to $\cl{T}$, there exists an $m$-ary repetition function $\phi_1$ such that $\phi_1 \leq \phi$, the set $F = \{\vect{a} \in A^m : \phi_1(\vect{a}) \neq 0\}$ is finite and $(\phi_1, S) \notin \cl{T}$. Observe that $S \neq B^m$, because otherwise $(\phi, S)$ would be a conjunctive minor of the $m$-ary generalized trivial constraint (by restricting the antecedent), which is in $\cl{T}$ by Lemma \ref{lemma:all}. Also, $\phi_1$ is not identically $0$, because otherwise $(\phi_1, S)$ would be a conjunctive minor of the $m$-ary generalized empty constraint (by extending the consequent), which is in $\cl{T}$ by Lemma \ref{lemma:all}. The set $\cl{T}$ cannot contain $(\phi_1, B^m \setminus \{\vect{s}\})$ for every $\vect{s} \in B^m \setminus S$, because if it did, then $(\phi_1, S)$ would be a conjunctive minor of the family $(\phi_1, B^m \setminus \{\vect{s}\})_{\vect{s} \in B^m \setminus S}$ (by intersecting consequents). Choose some $\vect{s} \in B^m \setminus S$ such that $(\phi_1, B^m \setminus \{\vect{s}\}) \notin \cl{T}$.

Consider the set $Q$ consisting of all $m$-ary repetition functions $\phi'$ on $A$ such that the restriction of $\phi'$ to $A^m \setminus F$ is identically $0$ and $(\phi', B^m \setminus \{\vect{s}\}) \in \cl{T}$. Since $\cl{T}$ is closed under taking the limit of antecedents, $Q$ is chain complete. Since the functions in $Q$ are completely determined by their restrictions to the finite set $F$ (they are all identically $0$ outside of $F$), we can apply Lemmas \ref{HellersteinLemma2.1} and \ref{HellersteinLemma2.2} to conclude that the set $Q_\mathrm{max}$ of maximal elements in $Q$ is finite, and for all $\phi' \in Q$, there exists $\phi'' \in Q_\mathrm{max}$ such that $\phi' \leq \phi''$.

Note that $(\phi_1, B^m \setminus \{\vect{s}\}) \notin \cl{T}$, and for all $\phi'' \in Q_\mathrm{max}$, $(\phi'', B^m \setminus \{\vect{s}\}) \in \cl{T}$. Therefore, for all $\phi'' \in Q_\mathrm{max}$, $\phi'' \not\geq \phi_1$.

The set $Q$ (and hence the set $Q_\mathrm{max}$) is not empty, because $\cl{T}$ contains $(\phi', B^m \setminus \{\vect{s}\})$ where $\phi'$ is identically $0$, which is a conjunctive minor of the $m$-ary generalized empty constraint (by extending the consequent), which is in $\cl{T}$ by Lemma \ref{lemma:all}.

Let $X = \{\vect{a} \in A^m : \phi_1(\vect{a}) \neq \infty\}$. Define an $m$-ary repetition function $\beta$ on $A$ such that for all $\vect{a} \in A^m$,
\begin{itemize}
\item $\beta(\vect{a}) = \phi_1(\vect{a})$, if $\vect{a} \in X$,
\item $\beta(\vect{a}) = 0$, if $\vect{a} \notin X$ and $\phi''(\vect{a}) = \infty$ for all $\phi'' \in Q_\mathrm{max}$,
\item $\beta(\vect{a}) = \max \{ \phi''(\vect{a}) + 1 : \text{$\phi'' \in Q_\mathrm{max}$ such that $\phi''(\vect{a}) \neq \infty$}\}$, otherwise.
\end{itemize}
In the third case, the value of $\beta$ is finite because $Q_\mathrm{max}$ is a finite set.

We claim that $(\beta, B^m \setminus \{\vect{s}\}) \notin \cl{T}$. To prove the claim, consider any $\phi'' \in Q_\mathrm{max}$. Since $\phi'' \not\geq \phi_1$, there exists an $\vect{a} \in A^m$ such that $\phi''(\vect{a}) < \phi_1(\vect{a})$, and hence $\phi''(\vect{a}) < \beta(\vect{a})$. Thus, there is no $\phi'' \in Q_\mathrm{max}$ such that $\beta \leq \phi''$, implying that $\beta \notin Q$. Therefore, $(\beta, B^m \setminus \{\vect{s}\}) \notin \cl{T}$.

Let $n = \sum_{\vect{a} \in A^m} \beta(\vect{a})$. Consider any $\phi'' \in Q_\mathrm{max}$. Because $\phi'' \not\geq \phi_1$, there exists an $\vect{a} \in A^m$ such that $\phi''(\vect{a}) < \phi_1(\vect{a})$ and hence $\beta(\vect{a}) > 0$. Therefore $n > 0$.

Let $\mtrx{D} = (\vect{d}^1, \dotsc, \vect{d}^n)$ be an $m \times n$ matrix whose columns consist of $\beta(\vect{a})$ copies of $\vect{a}$ for each $\vect{a} \in A^m$. Let $\mtrx{M} = (\vect{m}^1, \dotsc, \vect{m}^n)$ be a $\mu \times n$ matrix whose first $m$ rows are the rows of $\mtrx{D}$ (i.e., $\bigl( \vect{m}^1(i), \dotsc, \vect{m}^n(i) \bigr) = \bigl( \vect{d}^1(i), \dotsc, \vect{d}^n(i) \bigr)$ for every $i \in m$) and whose other rows are the remaining distinct $n$-tuples in $A^n$; every $n$-tuple in $A^n$ is a row of $\mtrx{M}$, and any repetition of rows can only occur in the first $m$ rows of $\mtrx{M}$. Note that $m \leq \mu$ and that $\mu$ is infinite if and only if $A$ is infinite. Let $\chi_\mtrx{M}$ be the characteristic function of $\mtrx{M}$, and let $S_\mtrx{M}$ be the $\mu$-ary relation consisting of those $\mu$-tuples $\vect{b} = (b_t \mid t \in \mu)$ in $B^\mu$ such that $(b_t \mid t \in m)$ belongs to $B^m \setminus \{\vect{s}\}$.

Observe that $(\chi_\mtrx{M}, S_\mtrx{M}) \notin \cl{T}^\infty$, because $(\beta, B^m \setminus \{\vect{s}\})$ is a simple $\infty$-minor of $(\chi_\mtrx{M}, S_\mtrx{M})$, and if $(\chi_\mtrx{M}, S_\mtrx{M}) \in \cl{T}^\infty$, we would conclude, from Fact \ref{fact:cmc}, that $(\beta, B^m \setminus \{\vect{s}\}) \in \cl{T}$. Furthermore, there must exist a $\mu$-tuple $\vect{u} = (u_t \mid t \in \mu)$ in $B^\mu$ such that $(u_t \mid t \in m) = \vect{s}$ and $(\chi_\mtrx{M}, B^\mu \setminus \{\vect{u}\}) \notin \cl{T}^\infty$; otherwise by arbitrary intersections of consequents we would conclude that $(\chi_\mtrx{M}, S_\mtrx{M}) \in \cl{T}^\infty$.

We can define an $n$-ary function $g$ by the condition $g \mtrx{M} = \vect{u}$. This definition is valid, because the set of rows of $\mtrx{M}$ is the set of all $n$-tuples in $A^n$, and if two rows of $\mtrx{M}$ coincide, then the corresponding components of $\vect{u}$ also coincide. For, suppose, on the contrary, that $\bigl( \vect{m}^1(i), \dotsc, \vect{m}^n(i) \bigr) = \bigl( \vect{m}^1(j), \dotsc, \vect{m}^n(j) \bigr)$ but $\vect{u}(i) \neq \vect{u}(j)$. Consider the $\mu$-ary generalized constraint $(\phi^=, S^=)$ from $A$ to $B$ defined by
\[
\phi^=(\vect{a}) =
\begin{cases}
\infty, & \text{if $a_i = a_j$,} \\
0, & \text{otherwise,}
\end{cases}
\qquad \text{and} \qquad
S^= = \{(b_t \mid t \in \mu) \in B^\mu : b_i = b_j\}.
\]
The generalized constraint $(\phi^=, S^=)$ is a simple $\infty$-minor of the binary generalized equality constraint and therefore belongs to $\cl{T}^\infty$. On the other hand, $(\chi_\mtrx{M}, B^\mu \setminus \{\vect{u}\})$ is a relaxation of $(\phi^=, S^=)$ and should also belong to $\cl{T}^\infty$, yielding the intended contradiction.

By the definition of $\vect{u}$, $g$ does not satisfy $(\chi_\mtrx{M}, S_\mtrx{M})$, and it is easily seen that $g$ does not satisfy $(\beta, B^m \setminus \{\vect{s}\})$. Since $\mtrx{N} \prec \phi$, $g$ does not satisfy $(\phi, S)$ either.

We then show that $g$ satisfies every generalized constraint in $\cl{T}^\infty$. Suppose, on the contrary, that there is a $\rho$-ary generalized constraint $(\phi_0, S_0) \in \cl{T}^\infty$, possibly infinitary, which is not satisfied by $g$. Thus, for some $\rho \times n$ matrix $\mtrx{M}_0 = (\vect{c}^1, \dotsc, \vect{c}^n) \prec \phi_0$ we have $g \mtrx{M}_0 \notin S_0$. Define $h \colon \rho \to \mu$ to be any map such that
\[
\bigl( \vect{c}^1(i), \dotsc, \vect{c}^n(i) \bigr) = \bigl( (\vect{m}^1 h)(i), \dotsc, (\vect{m}^n h)(i) \bigr)
\]
for every $i \in \rho$, i.e., row $i$ of $\mtrx{M}_0$ is the same as row $h(i)$ of $\mtrx{M}$, for each $i \in \rho$. Let $(\phi_h, S_h)$ be a $\mu$-ary simple $\infty$-minor of $(\phi_0, S_0)$ via $H = \{h\}$. Note that $(\phi_h, S_h) \in \cl{T}^\infty$.

We claim that $\chi_\mtrx{M} \leq \phi_h$. This will follow if we show that $\mtrx{M} \prec \phi_h$. To prove this, it is enough to show that $(\vect{m}^1 h,\dotsc, \vect{m}^n h) \prec \phi_0$. In fact, we have, for $1 \leq j \leq n$,
\[
\vect{m}^j h = (\vect{m}^j h(i) \mid i \in \rho) = (\vect{c}^j (i) \mid i \in \rho) = \vect{c}^j,
\]
and $(\vect{c}^1, \dotsc, \vect{c}^n) \prec \phi_0$.

Next we claim that $B^\mu \setminus \{\vect{u}\} \supseteq S_h$, i.e., $\vect{u} \notin S_h$. For that it is enough to show that $\vect{u} h \notin S_0$. For every $i \in \rho$ we have
\[
\begin{split}
(\vect{u} h)(i)
&= \bigl( g(\vect{m}^1, \dotsc, \vect{m}^n) h \bigr) (i) \\
&= g \bigl( (\vect{m}^1 h)(i), \dotsc, (\vect{m}^n h)(i) \bigr)
= g \bigl( \vect{c}^1(i), \dotsc, \vect{c}^n(i) \bigr).
\end{split}
\]
Thus $\vect{u} h = g \mtrx{M}_0$. Since $g \mtrx{M}_0 \notin S_0$, we conclude that $\vect{u} \notin S_h$.

So $(\chi_M, B^\mu \setminus \{\vect{u}\})$ is a relaxation of $(\phi_h, S_h)$, and we conclude that $(\chi_M, B^\mu \setminus \{\vect{u}\}) \in \cl{T}^\infty$. By the definition of $\vect{u}$, this is impossible. Thus we have proved Claim \ref{claim2}.

\smallskip
To see that the implication $\text{(i)} \Rightarrow \text{(ii)}$ holds, observe that, by Claim \ref{claim2}, for every generalized constraint $(\phi, S) \notin \cl{T}$, there is a function which does not satisfy $(\phi, S)$ but satisfies every generalized constraint in $\cl{T}$. Thus the set of all these ``separating'' functions constitutes the desired set characterizing $\cl{T}$.
\end{proof}

\section{Classes of operations closed under permutation of variables, addition of dummy variables, and composition}

We now consider the problem of characterizing the classes of operations on an arbitrary non-empty set $A$ that are closed under permutation of variables, addition of dummy variables, and composition (but not necessarily under identification of variables). In order to obtain a description resembling the one of clones in terms of relations or the one of classes closed under special minors in terms of generalized constraints, we need to confine ourselves to dealing only with classes that contain all projections, much in the same way as in the case of clones, which are iterative algebras containing all projections. In other words, we are going to characterize the subalgebras of $(\cl{O}_A; \zeta, \tau, \nabla, \ast)$ that contain all projections. Of course, every clone on $A$ is such a closed class. Examples of classes that are closed under the operations considered but not under identification of variables include:
\begin{itemize}
\item for a non-trivial clone $\cl{C}$ on $A$, the class of all projections and all operations in $\cl{C}$ with at least $n$ variables for some $n \geq 2$;
\item for a partial order $\leq$ on $A$, the class of functions that are order-preserving or order-reversing in each variable with respect to $\leq$;
\item for a field $A$, and for a fixed integer $p \geq 2$, the class of all linear functions $f(x_1, \dotsc, x_n) = \sum_{i=1}^n c_i x_i$, where $c_i \in A$ for $1 \leq i \leq n$, such that $\card{\{i : c_i \neq 0\}} \equiv 1 \pmod{p}$ (with the exception of the case $p = 2$ when $A$ is a two-element field).
\end{itemize}

For an integer $m \geq 1$, an $m$-ary \emph{cluster} on $A$ is an initial segment $\Phi$ of the set $\mathcal{M}(A^m)$ of all finite multisets on $A^m$, partially ordered by multiset inclusion ``$\subseteq$'', i.e., a subset $\Phi$ of $\mathcal{M}(A^m)$ such that, for all $S, T \in \mathcal{M}(A^m)$, if $S \in \Phi$ and $T \subseteq S$, then also $T \in \Phi$. The number $\max \{\card{S} : S \in \Phi\}$, if it exists, is called the \emph{breadth} of $\Phi$; if the maximum does not exist, then $\Phi$ is said to have \emph{infinite breadth.}

If $\mtrx{M} \in A^{m \times n}$ is an $m \times n$ matrix with entries from $A$ and $\Phi$ is an $m$-ary cluster on $A$, we write $\mtrx{M} \prec \Phi$ to mean that the multiset $\mtrx{M}^*$ of columns of $\mtrx{M}$ is an element of $\Phi$. If $f \colon A^n \to A$ is an $n$-ary operation on $A$ and $\Phi$ is an $m$-ary cluster on $A$, we say that $f$ \emph{satisfies} $\Phi$, if for every matrix $\mtrx{M}$ it holds that whenever $\mtrx{M} \prec \Phi$ and $\mtrx{M} = [\mtrx{M}_1 | \mtrx{M}_2]$ where $\mtrx{M}_1$ has $n$ columns and $\mtrx{M}_2$ may be empty, we have that $[f \mtrx{M}_1 | \mtrx{M}_2] \prec \Phi$.

We say that a class $\cl{F}$ of operations on $A$ is \emph{characterized} by a set $\cl{T}$ of clusters on $A$, if $\cl{F}$ is precisely the class of operations on $A$ that satisfy every cluster in $\cl{T}$. Similarly, we say that $\cl{T}$ is \emph{characterized} by $\cl{F}$, if $\cl{T}$ is precisely the set of clusters that are satisfied by every operation in $\cl{F}$.

\begin{remark}
\label{rem:msrep}
Recalling that a finite multiset $S$ on $A^m$ is a map $\nu_S \colon A^m \to \mathbb{N}$, it is in fact an $m$-ary repetition function on $A$. Thus, adopting the notation for matrices and repetition functions introduced in Section \ref{sec:functions}, for a matrix $\mtrx{M} \in A^{m \times n}$, we write $\mtrx{M} \prec S$ to mean that each $m$-tuple $\vect{a} \in A^m$ occurs as a column of $\mtrx{M}$ at most as many times as the multiplicity $\nu_S(\vect{a})$ indicates, i.e., $\chi_\mtrx{M}(\vect{a}) \leq \nu_S(\vect{a})$ for all $\vect{a} \in A^m$, i.e., $\mtrx{M}^* \subseteq S$.
\end{remark}

\begin{remark}
Alternatively, we can define an $m$-ary cluster $\Phi$ on $A$ to be a set of $m$-ary repetition functions on $A$. Then $\mtrx{M} \prec \Phi$ means that $\mtrx{M} \prec \phi$ for some $\phi \in \Phi$ (see Section \ref{sec:functions}). To see that these definitions are equivalent, observe first that every set of finite multisets is in fact itself a set of repetition functions (cf.\ Remark \ref{rem:msrep}). On the other hand, a set $\Phi_\mathrm{R}$ of repetition functions corresponds to the downward closed set $\Phi_\mathrm{F}$ of finite multisets $S$ satisfying $\nu_S(\vect{a}) \leq \phi(\vect{a})$ for all $\vect{a} \in A^m$, for some $\phi \in \Phi_\mathrm{R}$. It can be easily shown that if $\Phi_\mathrm{F}$ is a downward closed set of multisets and $\Phi_\mathrm{R}$ is a set of repetition functions such that $\Phi_\mathrm{F}$ and $\Phi_\mathrm{R}$ correspond to each other under the two alternative definitions of cluster, then $\mtrx{M} \prec \Phi_\mathrm{F}$ if and only if $\mtrx{M} \prec \Phi_\mathrm{R}$.

While we keep to the original definition of cluster when we prove our theorems, we may sometimes find it simpler to represent clusters in terms of repetition functions in the subsequent examples.
\end{remark}

\begin{example}
An $m$-ary relation $R$ on $A$ is equivalent to the $m$-ary cluster
\[
\Phi_R = \bigl\{ S \in \mathcal{M}(A^m) : \forall \vect{a} \in A^m \bigl( \nu_S(\vect{a}) > 0 \Rightarrow \vect{a} \in R \bigr) \bigr\}.
\]
Thus every locally closed clone can be characterized by a set of clusters of this kind.

Using the alternative definition of cluster, $\Phi_R$ is equivalent to the cluster $\{\phi_R\}$, where the repetition function $\phi_R$ is defined by the rule $\phi_R(\vect{a}) = \infty$ if $\vect{a} \in R$ and $\phi_R(\vect{a}) = 0$ otherwise.
\end{example}

\begin{example}
Let $\leq$ be a partial order on $A$. A function $f \colon A^n \to A$ is not order-preserving nor order-reversing in its $i$-th variable if there exist elements $a_1, \dotsc, a_n, a'_i, b_1, \dotsc, b_n, b'_i \in A$ such that $a_i < a'_i$, $b'_i < b_i$ and
\begin{align*}
f(a_1, \dotsc, a_{i-1}, a_i, a_{i+1}, \dotsc, a_n) &< f(a_1, \dotsc, a_{i-1}, a'_i, a_{i+1}, \dotsc, a_n), \\
f(b_1, \dotsc, b_{i-1}, b_i, b_{i+1}, \dotsc, b_n) &> f(b_1, \dotsc, b_{i-1}, b'_i, b_{i+1}, \dotsc, b_n).
\end{align*}
Thus, it is easy to see that the class of operations on $A$ that are order-preserving or order-reversing in each variable with respect to $\leq$ is characterized by the quaternary cluster $\Phi_\leq$ consisting precisely of those finite multisets $S$ on $A^4$ that satisfy the conditions
\begin{itemize}
\item $\nu_S(a, b, c, d) = 0$ whenever $a < b$ and $c > d$, or $a > b$ and $c < d$, or $a$ and $b$ are incomparable, or $c$ and $d$ are incomparable; and
\item for $X = \{(a, b, c, d) \in A^4 : (a \leq b \wedge c \leq d) \vee (a \geq b \wedge c \geq d) \wedge (a \neq b \vee c \neq d)\}$, $\sum_{\vect{a} \in X} \nu_S(\vect{a}) \leq 1$.
\end{itemize}
\end{example}

\begin{theorem}
Let $A$ be an arbitrary, possibly infinite non-empty set. For any class $\cl{F}$ of operations on $A$, the following two conditions are equivalent:
\begin{compactenum}[(i)]
\item $\cl{F}$ is locally closed, contains all projections, and is closed under permutation of variables, addition of dummy variables, and composition.
\item $\cl{F}$ is characterized by a set $\cl{T}$ of clusters on $A$.
\end{compactenum}
\end{theorem}
\begin{proof}
$\text{(ii)} \Rightarrow \text{(i)}$: It is straightforward to verify that the class of operations satisfying a set of clusters is closed under permutation of variables and addition of dummy variables, and it contains all projections. To see that it is closed under composition, let $f \in \cl{F}^{(n)}$ and $g \in \cl{F}^{(p)}$, and consider $f \ast g \colon A^{n + p - 1} \to A$. Let $\Phi \in \cl{T}$, and let $\mtrx{M} \prec \Phi$ be a matrix such that $\mtrx{M} = [\mtrx{M}_1 | \mtrx{M}_2 | \mtrx{M}_3]$, where $\mtrx{M}_1$ has $p$ columns and $\mtrx{M}_2$ has $n-1$ columns. Since $g$ satisfies $\Phi$, we have that $[g \mtrx{M}_1 | \mtrx{M}_2 | \mtrx{M}_3] \prec \Phi$. Then $[g \mtrx{M}_1 | \mtrx{M}_2]$ has $n$ columns, and since $f$ satisfies $\Phi$, we have that $\bigl[ f[g \mtrx{M}_1 | \mtrx{M}_2] \big| \mtrx{M}_3 \bigr] \prec \Phi$. But $f[g \mtrx{M}_1 | \mtrx{M}_2] = (f \ast g)[\mtrx{M}_1 | \mtrx{M}_2]$, so $\bigl[ (f \ast g)[\mtrx{M}_1 | \mtrx{M}_2] \big| \mtrx{M}_3 \bigr] \prec \Phi$, and we conclude that $f \ast g$ satisfies $\Phi$.

To show that $\cl{F}$ is locally closed, consider $g \notin \cl{F}$ and let $\Phi \in \cl{T}$ be a cluster that is not satisfied by $g$ (but it is satisfied by every operation in $\cl{F}$). Assume that $g$ is $n$-ary. Thus, for some matrix $\mtrx{M} = [\mtrx{M}_1 | \mtrx{M}_2] \prec \Phi$ where $\mtrx{M}_1$ has $n$ columns, we have that $[g \mtrx{M}_1 | \mtrx{M}_2] \nprec \Phi$, but $[f \mtrx{M}_1 | \mtrx{M}_2] \prec \Phi$ for every $n$-ary operation $f$ in $\cl{F}$. Thus, the restriction of $g$ to the finite set of rows of $\mtrx{M}_1$ does not coincide with that of any member of $\cl{F}$.

\smallskip
$\text{(i)} \Rightarrow \text{(ii)}$: We need to show that for every operation $g \notin \cl{F}$, there exists a cluster $\Phi$ that is satisfied by every operation in $\cl{F}$ but not by $g$. The set of all such ``separating'' clusters, for each $g \notin \cl{F}$, characterizes $\cl{F}$.

Since $\cl{F}$ contains all projections, $\cl{F} \neq \emptyset$. Suppose that $g \notin \cl{F}$ is $n$-ary. Since $\cl{F}$ is locally closed, there is a finite subset $F \subseteq A^n$ such that $g|_F \neq f|_F$ for every $f \in \cl{F}^{(n)}$. Clearly $F$ is non-empty. Let $\mtrx{M}$ be a $\card{F} \times n$ matrix whose rows are the elements of $F$ in some fixed order. Recall that $\mtrx{M}^*$ denotes the multiset of columns of $\mtrx{M}$.

Let $X$ be any submultiset of $\mtrx{M}^*$. Let $\Pi = (\mtrx{M}_1, \dotsc, \mtrx{M}_q)$ be a sequence of submatrices of $\mtrx{M}$ such that $\{\mtrx{M}_1^*, \dotsc, \mtrx{M}_q^*\}$ is a partition of $\mtrx{M}^* \setminus X$. For $1 \leq i \leq q$, let $\vect{d}^i \in \cl{F} \mtrx{M}_i$, and let $\mtrx{D} = (\vect{d}^1, \dotsc, \vect{d}^q)$. (Note that each $\cl{F} \mtrx{M}_i$ is non-empty, because $\cl{F}$ contains all projections. Observe also that each $\cl{F} \mtrx{M}_i$ is a subset of $\cl{F} \mtrx{M}$, because $\cl{F}$ is closed under addition of dummy variables.) Denote $\langle X, \Pi, \mtrx{D} \rangle = X \uplus \mtrx{D}^*$.

We define $\Phi$ to be the set of all submultisets of the multisets $\langle X, \Pi, \mtrx{D} \rangle$ for all possible choices of $X$, $\Pi$, and $\mtrx{D}$. Observe first that $g$ does not satisfy $\Phi$. For, it holds that $\mtrx{M} \prec \Phi$, because $\mtrx{M}^* = \langle \mtrx{M}^*,(),() \rangle \in \Phi$. On the other hand, since $g \mtrx{M} \notin \cl{F} \mtrx{M}$, we have that $g \mtrx{M} \notin \langle X, \Pi, \mtrx{D} \rangle$ for all $X$, $\Pi$, $\mtrx{D}$, and hence $g \mtrx{M} \nprec \Phi$.

\begin{clm}
\label{clm:partitions}
Let $X$ be a submultiset of $\mtrx{M}^*$, let $\Pi = (\mtrx{M}_1, \dotsc, \mtrx{M}_q)$ be a sequence of submatrices of $\mtrx{M}$ such that $\{\mtrx{M}_1^*, \dotsc, \mtrx{M}_q^*\}$ is a partition of $\mtrx{M}^* \setminus X$, and let $\mtrx{D} = (\vect{d}^1, \dotsc, \vect{d}^q)$, where each $\vect{d}^i \in \cl{F} \mtrx{M}_i$. If $\mtrx{N} = [\mtrx{N}_1 | \mtrx{N}_2] \prec \langle X, \Pi, \mtrx{D} \rangle$, where $\mtrx{N}_1$ has $n$ columns, then for all $f \in \cl{F}^{(n)}$, there exist $X'$, $\Pi'$, $\mtrx{D}'$ such that $[f \mtrx{N}_1 | \mtrx{N}_2] \prec \langle X', \Pi', \mtrx{D}' \rangle$.
\end{clm}

\noindent\textit{Proof of Claim \ref{clm:partitions}.}
By induction on $q$.
If $q = 0$, then $X = \mtrx{M}^*$, $\Pi = ()$, $\mtrx{D} = ()$, and $\langle X, \Pi, \mtrx{D} \rangle = \mtrx{M}^*$, and the condition $\mtrx{N} = [\mtrx{N}_1 | \mtrx{N}_2] \prec \langle \mtrx{M}^*, (), () \rangle$ means that $\mtrx{N}$ is a submatrix of $\mtrx{M}$. Then $f \mtrx{N}_1 \in \cl{F} \mtrx{N}_1$ and $[f \mtrx{N}_1 | \mtrx{N}_2] \prec \langle \mtrx{M}^* \setminus \mtrx{N}_1^*, (\mtrx{N}_1), (f \mtrx{N}_1) \rangle$.

Assume that the claim holds for $q = k \geq 0$, and consider the case that $q = k + 1$. Let $\mtrx{N} = [\mtrx{N}_1 | \mtrx{N}_2] \prec \langle X, \Pi, \mtrx{D} \rangle$. If $\mtrx{N}_1 \prec X$, then $f \mtrx{N}_1 \in \cl{F} \mtrx{N}_1$ and
\[
[f \mtrx{N}_1 | \mtrx{N}_2] \prec \langle X \setminus \mtrx{N}_1^*, (\mtrx{M}_1, \dotsc, \mtrx{M}_{k+1}, \mtrx{N}_1), (\vect{d}^1, \dotsc, \vect{d}^{k+1}, f \mtrx{N}_1) \rangle.
\]
Otherwise, for some $i \in \{1, \dotsc, k+1\}$, $\vect{d}^i$ is a column of $\mtrx{N}_1$. Denote by $\mtrx{N}'_1$ the matrix obtained from $\mtrx{N}_1$ by deleting the column $\vect{d}^i$. By a suitable permutation of variables, we obtain an operation $f' \in \cl{F}$ such that $f \mtrx{N}_1 = f' [\vect{d}^i | \mtrx{N}'_1]$. There is an operation $h \in \cl{F}$ such that $h \mtrx{M}_i = \vect{d}^i$, and we have that
\[
f' [\vect{d}^i | \mtrx{N}'_1]
= f' [h \mtrx{M}_i | \mtrx{N}'_1]
= (f' \ast h) [\mtrx{M}_i | \mtrx{N}'_1].
\]
Since $\cl{F}$ is closed under composition, $f' \ast h \in \cl{F}$. Furthermore,
\begin{multline*}
[\mtrx{M}_i | \mtrx{N}'_1] \prec \langle X \uplus \mtrx{M}_i^*,
(\mtrx{M}_1, \dotsc, \mtrx{M}_{i-1}, \mtrx{M}_{i+1}, \dotsc, \mtrx{M}_{k+1}), \\
(\vect{d}^1, \dotsc, \vect{d}^{i-1}, \vect{d}^{i+1}, \dotsc, \vect{d}^{k+1}) \rangle.
\end{multline*}
By the induction hypothesis, there exist $X'$, $\Pi'$, $\mtrx{D}'$ such that $\bigl[ (f' \ast h) [\mtrx{M}_i | \mtrx{N}'_1] \big| \mtrx{N}_2 \bigr] \prec \langle X', \Pi', \mtrx{D}' \rangle$, and hence $[f \mtrx{N}_1 | \mtrx{N}_2] \prec \langle X', \Pi', \mtrx{D}' \rangle$. This completes the proof of Claim \ref{clm:partitions}.

\smallskip
It follows from Claim \ref{clm:partitions} that every operation in $\cl{F}$ satisfies $\Phi$. This completes the proof of the theorem.
\end{proof}

\section{Closure conditions for clusters}

In order to describe the sets of clusters that are characterized by classes of operations, we need to introduce a number of operations on clusters. First, we will adapt the notion of conjunctive minor to clusters.

Let $H = (h_j)_{j \in J}$ be a minor formation scheme with target $m$, indeterminate set $V$, and source family $(n_j)_{j \in J}$. Let $(\Phi_j)_{j \in J}$ be a family of clusters on $A$, each $\Phi_j$ of arity $n_j$, and let $\Phi$ be an $m$-ary cluster on $A$. We say that $\Phi$ is a \emph{conjunctive minor} of the family $(\Phi_j)_{j \in J}$ \emph{via $H$,} or simply a \emph{conjunctive minor} of the family $(\Phi_j)_{j \in J}$, if, for every $m \times n$ matrix $\mtrx{M} = (\vect{a}^1, \dotsc, \vect{a}^n) \in A^{m \times n}$, the condition $\mtrx{M} \prec \Phi$ is equivalent to the condition that there are Skolem maps $\sigma_i \colon V \to A$, $1 \leq i \leq n$, such that, for all $j \in J$, we have $\bigl( (\vect{a}^1 + \sigma_1) h_j, \dotsc, (\vect{a}^n + \sigma_n) h_j \bigr) \prec \Phi_j$. If the minor formation scheme $H = (h_j)_{j \in J}$ and the family $(\Phi_j)_{j \in J}$ are indexed by a singleton $J = \{0\}$, then a conjunctive minor $\Phi$ of a family containing a single cluster $\Phi_0$ is called a \emph{simple minor} of $\Phi_0$.

The formation of conjunctive minors subsumes the formation of simple minors and the intersection of clusters. Simple minors in turn subsume permutation of arguments, projection, identification of arguments, and addition of a dummy argument, operations which can be defined for clusters in an analogous way as for generalized constraints.

\begin{lemma}
\label{lemma:tightcm}
Let $\Phi$ be a conjunctive minor of a non-empty family $(\Phi_j)_{j \in J}$ of clusters on $A$. If $f \colon A^n \to A$ satisfies $\Phi_j$ for all $j \in J$, then $f$ satisfies $\Phi$.
\end{lemma}
\begin{proof}
Let $\Phi$ be an $m$-ary conjunctive minor of the family $(\Phi_j)_{j \in J}$ via the scheme $H = (h_j)_{j \in J}$, $h_j \colon n_j \to m \cup V$. Let $\mtrx{M} = (\vect{a}^1, \dotsc, \vect{a}^{n'})$ be an $m \times n'$ matrix ($n' \geq n$) such that $\mtrx{M} \prec \Phi$, and denote $\mtrx{M}_1 = (\vect{a}^1, \dotsc, \vect{a}^n)$, $\mtrx{M}_2 = (\vect{a}^{n+1}, \dotsc, \vect{a}^{n'})$, so $\mtrx{M} = [\mtrx{M}_1 | \mtrx{M}_2]$. We need to prove that $[f \mtrx{M}_1 | \mtrx{M}_2] \prec \Phi$.

Since $\Phi$ is a conjunctive minor of $(\Phi_j)_{j \in J}$ via $H = (h_j)_{j \in J}$, there are Skolem maps $\sigma_i \colon V \to A$, $1 \leq i \leq n'$, such that for every $j \in J$, we have
\[
\bigl( (\vect{a}^1 + \sigma_1) h_j, \dotsc, (\vect{a}^{n'} + \sigma_{n'}) h_j \bigr) \prec \Phi_j.
\]
Denote
\begin{align*}
\mtrx{M}^j_1 &= \bigl( (\vect{a}^1 + \sigma_1) h_j, \dotsc, (\vect{a}^n + \sigma_n) h_j \bigr), \\
\mtrx{M}^j_2 &= \bigl( (\vect{a}^{n+1} + \sigma_{n+1}) h_j, \dotsc, (\vect{a}^{n'} + \sigma_{n'}) h_j \bigr).
\end{align*}
Since $f$ is assumed to satisfy  $\Phi_j$, we have that $[f \mtrx{M}^j_1 | \mtrx{M}^j_2] \prec \Phi_j$ for each $j \in J$.

Let $\sigma = f(\sigma_1, \dotsc, \sigma_n)$. We have that, for each $j \in J$,
\begin{align*}
(f \mtrx{M}_1 + \sigma) h_j
&= \bigl( f(\vect{a}^1, \dotsc, \vect{a}^n) + f(\sigma_1, \dotsc, \sigma_n) \bigr) h_j \\
&= \bigl( f(\vect{a}^1 + \sigma_1, \dotsc, \vect{a}^n + \sigma_n) \bigr) h_j \\
&= f \bigl( (\vect{a}^1 + \sigma_1) h_j, \dotsc, (\vect{a}^n + \sigma_n) h_j \bigr)
= f \mtrx{M}^j_1.
\end{align*}
Since $\Phi$ is a conjunctive minor of $(\Phi_j)_{j \in J}$ via $H = (h_j)_{j \in J}$ and
\[
\bigl( (f \mtrx{M}_1 + \sigma) h_j, (\vect{a}^{n+1} + \sigma_{n+1}) h_j, \dotsc, (\vect{a}^{n'} + \sigma_{n'}) h_j \bigr) = [f \mtrx{M}^j_1 | \mtrx{M}^j_2] \prec \Phi_j
\]
for each $j \in J$, we have that $[f \mtrx{M}_1 | \mtrx{M}_2] \prec \Phi$. Thus $f$ satisfies $\Phi$.
\end{proof}

\begin{lemma}
\label{lemma:unions}
Let $(\Phi_j)_{j \in J}$ be a non-empty family of $m$-ary clusters on $A$. If $f \colon A^n \to A$ satisfies $\Phi_j$ for all $j \in J$, then $f$ satisfies $\bigcup_{j \in J} \Phi_j$.
\end{lemma}
\begin{proof}
Let $\mtrx{M} = [\mtrx{M}_1 | \mtrx{M}_2] \prec \bigcup_{j \in J} \Phi_j$. Then $\mtrx{M} \prec \Phi_j$ for some $j \in J$. Since $f$ is assumed to satisfy $\Phi_j$, we have that $[f \mtrx{M}_1 | \mtrx{M}_2] \prec \Phi_j$, and hence $[f \mtrx{M}_1 | \mtrx{M}_2] \prec \bigcup_{j \in J} \Phi_j$.
\end{proof}

The \emph{quotient} of an $m$-ary cluster $\Phi$ on $A$ with a multiset $S \in \mathcal{M}(A^m)$ is defined as
\[
\Phi / S = \{S' \in \mathcal{M}(A^m) : S \uplus S' \in \Phi\}.
\]
It is easy to see that $\Phi / S$ is a cluster and $\Phi / S \subseteq \Phi$ for any $S$.

\begin{lemma}
\label{lemma:quotients}
Let $\Phi$ be an $m$-ary cluster on $A$. If $f$ satisfies $\Phi$, then $f$ satisfies $\Phi / S$ for every multiset $S$ on $A^m$.
\end{lemma}
\begin{proof}
Let $[\mtrx{M}_1 | \mtrx{M}_2] \prec \Phi / S$. Let $\mtrx{N}$ be a matrix with $\mtrx{N}^* = S$. Then $[\mtrx{M}_1 | \mtrx{M}_2 | \mtrx{N}] \prec \Phi$, and since $f$ is assumed to satisfy $\Phi$, we have that $[f \mtrx{M}_1 | \mtrx{M}_2 | \mtrx{N}] \prec \Phi$. Thus, $[f \mtrx{M}_1 | \mtrx{M}_2] \prec \Phi / S$, and we conclude that $f$ satisfies $\Phi / S$.
\end{proof}

\begin{lemma}
\label{lemma:dividends}
Assume that $\Phi$ is an $m$-ary cluster on $A$ that contains all multisets on $A^m$ of cardinality at most $p$. If $f$ satisfies all quotients $\Phi / S$ where $\card{S} \geq p$, then $f$ satisfies $\Phi$.
\end{lemma}
\begin{proof}
Let $f \colon A^n \to A$. Let $[\mtrx{M}_1 | \mtrx{M}_2] \prec \Phi$, where $\mtrx{M}_1$ has $n$ columns and $\mtrx{M}_2$ has $n'$ columns. If $n' < p$, then the number of columns of $[f \mtrx{M}_1 | \mtrx{M}_2]$ is $n' + 1 \leq p$, and hence $[f \mtrx{M}_1 | \mtrx{M}_2] \prec \Phi$. Otherwise $n' \geq p$ and, by our assumption, $f$ satisfies $\Phi / \mtrx{M}_2^*$. Thus, since $\mtrx{M}_1 \prec \Phi / \mtrx{M}_2^*$, we have that $f \mtrx{M}_1 \prec \Phi / \mtrx{M}_2^*$. Therefore $[f \mtrx{M}_1 | \mtrx{M}_2] \prec \Phi$, and we conclude that $f$ satisfies $\Phi$.
\end{proof}

For $p \geq 0$, the $m$-ary \emph{trivial cluster of breadth $p$,} denoted $\Omega_m^{(p)}$, is the set of all finite multisets on $A^m$ of cardinality at most $p$. The $m$-ary \emph{empty cluster} on $A$ is the empty set $\emptyset$. Note that $\Omega_m^{(0)} \neq \emptyset$, because the empty multiset $\varepsilon$ on $A^m$ is the unique member of $\Omega_m^{(0)}$. The binary \emph{equality cluster} on $A$, denoted $E_2$, is the set of all finite multisets $S$ on $A^2$ for which it holds that $\nu_S(\vect{a}) = 0$ whenever $\vect{a} = (a, b)$ with $a \neq b$.

For $p \geq 0$, we say that the cluster $\Phi^{(p)} = \Phi \cap \Omega_m^{(p)}$ is obtained from the $m$-ary cluster $\Phi$ by \emph{restricting the breadth to $p$.}

\begin{lemma}
Let $\Phi$ be an $m$-ary cluster on $A$. Then $f$ satisfies $\Phi$ if and only if $f$ satisfies $\Phi^{(p)}$ for all $p \geq 0$.
\end{lemma}
\begin{proof}
Assume first that $f$ satisfies $\Phi$. Let $[\mtrx{M}_1 | \mtrx{M}_2] \prec \Phi^{(p)}$. Since $\Phi^{(p)} \subseteq \Phi$, we have that $[\mtrx{M}_1 | \mtrx{M}_2] \prec \Phi$, and hence $[f \mtrx{M}_1 | \mtrx{M}_2] \prec \Phi$ by our assumption. The number of columns of $[f \mtrx{M}_1 | \mtrx{M}_2]$ is at most $p$, so we have that $[f \mtrx{M}_1 | \mtrx{M}_2] \prec \Phi^{(p)}$. Thus, $f$ satisfies $\Phi^{(p)}$.

Assume then that $f$ satisfies $\Phi^{(p)}$ for all $p \geq 0$. Let $M = [\mtrx{M}_1 | \mtrx{M}_2] \prec \Phi$, and let $q$ be the number of columns in $\mtrx{M}$. Then $[\mtrx{M}_1 | \mtrx{M}_2] \prec \Phi^{(q)}$, and hence $[f \mtrx{M}_1 | \mtrx{M}_2] \prec \Phi^{(q)}$ by our assumption. Since $\Phi^{(p)} \subseteq \Phi$, we have that $[f \mtrx{M}_1 | \mtrx{M}_2] \prec \Phi$, and we conclude that $f$ satisfies $\Phi$.
\end{proof}

We say that a set $\cl{T}$ of clusters on $A$ is \emph{closed under quotients,} if for any $\Phi \in \cl{T}$, every quotient $\Phi / S$ is also in $\cl{T}$. We say that $\cl{T}$ is \emph{closed under dividends,} if for every cluster $\Phi$ it holds that $\Phi \in \cl{T}$ whenever $\Omega_m^{(p)} \subseteq \Phi$ and $\Phi / S \in \cl{T}$ for every multiset $S$ on $A^m$ of cardinality at least $p$. We say that $\cl{T}$ is \emph{locally closed,} if $\Phi \in \cl{T}$ whenever $\Phi^{(p)} \in \cl{T}$ for all $p \geq 0$. We say than $\cl{T}$ is \emph{closed under unions,} if $\bigcup_{j \in J} \Phi_j \in \cl{T}$ whenever $(\Phi_j)_{j \in J}$ is a family of $m$-ary clusters in $\cl{T}$. We say that $\cl{T}$ is \emph{closed under formation of conjunctive minors,} if all conjunctive minors of non-empty families of members of $\cl{T}$ are members of $\cl{T}$.

\begin{theorem}
Let $A$ be an arbitrary, possibly infinite non-empty set. For any set $\cl{T}$ of clusters on $A$, the following two conditions are equivalent:
\begin{compactenum}[(i)]
\item $\cl{T}$ is locally closed and contains the binary equality cluster, the unary empty cluster, and all unary trivial clusters of breadth $p \geq 0$, and it is closed under formation of conjunctive minors, unions, quotients, and dividends.
\item $\cl{T}$ is characterized by some set of operations on $A$.
\end{compactenum}
\end{theorem}
\begin{proof}
$\text{(ii)} \Rightarrow \text{(i)}$: It is clear that every function satisfies the equality, empty, and trivial clusters. By Lemmas \ref{lemma:tightcm}, \ref{lemma:unions}, \ref{lemma:quotients}, and \ref{lemma:dividends}, $\cl{T}$ is closed under formation of conjunctive minors, unions, quotients, and dividends. In order to show that $\cl{T}$ is locally closed, let $\Phi$ be an $m$-ary cluster not in $\cl{T}$. By (ii), there is a function $f$ that satisfies every cluster in $\cl{T}$ but does not satisfy $\Phi$. Thus, there is an $m \times p$ matrix $\mtrx{M} = [\mtrx{M}_1 | \mtrx{M}_2] \prec \Phi$ such that $[f \mtrx{M}_1 | \mtrx{M}_2] \nprec \Phi$. Restricting the breadth of $\Phi$ to $p$, i.e., taking the intersection $\Phi^{(p)} = \Phi \cap \Omega_m^{(p)}$, we have that $[\mtrx{M}_1 | \mtrx{M}_2] \prec \Phi^{(p)}$ and $[f \mtrx{M}_1 | \mtrx{M}_2] \nprec \Phi^{(p)}$. Thus $f$ does not satisfy $\Phi^{(p)}$, so $\Phi^{(p)} \notin \cl{T}$. This completes the proof of the implication $\text{(ii)} \Rightarrow \text{(i)}$.

\smallskip
$\text{(i)} \Rightarrow \text{(ii)}$: Similarly to the proof of Theorem \ref{thm:gc}, we extend the notion of cluster to arbitrary, possibly infinite arities. We use the terms \emph{conjunctive $\infty$-minor} and \emph{simple $\infty$-minor} to refer to conjunctive minors and simple minors via a scheme whose target and source ordinals may be infinite or finite. The use of the term ``minor'' without the prefix ``$\infty$'' continues to mean the respective minor via a scheme whose target and source ordinals are all finite. Matrices can also have infinitely many rows but only a finite number of columns.

For a set $\cl{T}$ of clusters on $A$, we denote by $\cl{T}^\infty$ the set of those clusters which are conjunctive $\infty$-minors of families of members of $\cl{T}$. This set $\cl{T}^\infty$ is the smallest set of clusters containing $\cl{T}$ which is closed under formation of conjunctive $\infty$-minors, and it is called the \emph{conjunctive $\infty$-minor closure} of $\cl{T}$. Analogously to the proof of Theorem \ref{thm:gc}, considering the formation of repeated conjunctive $\infty$-minors, we can show that the following holds.

\begin{fact}
\label{fact:clustermc}
Let $\cl{T}$ be a set of finitary clusters on $A$, and let $\cl{T}^\infty$ be its conjunctive $\infty$-minor closure. If $\cl{T}$ is closed under formation of conjunctive minors, then $\cl{T}$ is the set of all finitary clusters belonging to $\cl{T}^\infty$.
\end{fact}

Let $\cl{T}$ be a set of finitary clusters satisfying the conditions of (i). For every finitary cluster $\Phi \notin \cl{T}$, we shall construct a function $g$ that satisfies all clusters in $\cl{T}$ but does not satisfy $\Phi$. Thus, the set of these ``separating'' functions, for all $\Phi \notin \cl{T}$, constitutes a set characterizing $\cl{T}$.

Thus, let $\Phi$ be a finitary cluster on $A$ that is not in $\cl{T}$. Note that, by Fact \ref{fact:clustermc}, $\Phi$ cannot be in $\cl{T}^\infty$. Let $m$ be the arity of $\Phi$. Since $\cl{T}$ is locally closed and $\Phi$ does not belong to $\cl{T}$, there is an integer $p$ such that $\Phi^{(p)} = \Phi \cap \Omega_m^{(p)} \notin \cl{T}$; let $n$ be the smallest such integer. Every function that does not satisfy $\Phi^{(n)}$ does not satisfy $\Phi$ either, so we can consider $\Phi^{(n)}$ instead of $\Phi$. Due to the minimality of $n$, the breadth of $\Phi^{(n)}$ is $n$. Observe that $\Phi$ is not the trivial cluster of breadth $n$ nor the empty cluster, because these are members of $\cl{T}$. Thus, $n \geq 1$.

We can assume that $\Phi$ is a minimal nonmember of $\cl{T}$ with respect to deletion of rows (projection), i.e., every simple minor of $\Phi$ obtained by deleting some rows of $\Phi$ is a member of $\cl{T}$. If this is not the case, then we can delete some rows of $\Phi$ to obtain a minimal nonmember $\Phi'$ of $\cl{T}$ and consider the cluster $\Phi'$ instead of $\Phi$. Note that by Lemma \ref{lemma:tightcm}, every function not satisfying $\Phi'$ does not satisfy $\Phi$ either.

We can also assume that $\Phi$ is a minimal nonmember of $\cl{T}$ with respect to quotients, i.e., whenever $S \neq \varepsilon$, we have that $\Phi / S \in \cl{T}$. If this is not the case, then consider a minimal nonmember $\Phi / S$ of $\cl{T}$ instead of $\Phi$. By Lemma \ref{lemma:quotients}, every function not satisfying $\Phi / S$ does not satisfy $\Phi$ either.

The fact that $\Phi$ is a minimal nonmember of $\cl{T}$ with respect to quotients implies that $\Omega_m^{(1)} \not\subseteq \Phi$. For, suppose, on the contrary, that $\Omega_m^{(1)} \subseteq \Phi$. Since all quotients $\Phi / S$ where $\card{S} \geq 1$ are in $\cl{T}$ and $\cl{T}$ is closed under dividends, we have that $\Phi \in \cl{T}$, a contradiction.

Let $\Psi = \bigcup \{\Phi' \in \cl{T} : \Phi' \subseteq \Phi\}$, i.e., $\Psi$ is the largest cluster in $\cl{T}$ such that $\Psi \subseteq \Phi$. Note that this is not the empty union, because the empty cluster is a member of $\cl{T}$. It is clear that $\Psi \neq \Phi$. Since $n$ was chosen to be the smallest integer satisfying $\Phi^{(n)} \notin \cl{T}$, we have that $\Phi^{(n-1)} \in \cl{T}$ and since $\Phi^{(n-1)} \subseteq \Phi^{(n)}$, it holds that $\Phi^{(n-1)} \subseteq \Psi$. Thus there is a multiset $Q \in \Phi \setminus \Psi$ with $\card{Q} = n$. Let $\mtrx{D} = (\vect{d}^1, \dotsc, \vect{d}^n)$ be an $m \times n$ matrix whose multiset of columns equals $Q$.

The rows of $\mtrx{D}$ are pairwise distinct. For, suppose, for the sake of contradiction, that rows $i$ and $j$ of $\mtrx{D}$ coincide. Since $\Phi$ is a minimal nonmember of $\cl{T}$ with respect to deleting rows, by deleting row $j$ of $\Phi$, we obtain a cluster $\Phi'$ that is in $\cl{T}$. By adding a dummy row in the place of the deleted row, and finally by intersecting with the conjunctive minor of the binary equality cluster whose $i$-th and $j$-th rows are equal, we obtain a cluster in $\cl{T}$ that contains $Q$ and is a subset of $\Phi$. But this is impossible by the choice of $Q$.

Let $\Upsilon = \bigcap \{\Phi' \in \cl{T} : Q \in \Phi'\}$, i.e., $\Upsilon$ is the smallest cluster in $\cl{T}$ that contains $Q$ as an element. Note that this is not the empty intersection, because the trivial cluster $\Omega_m^{(n)}$ is a member $\cl{T}$ and contains $Q$. By the choice of $Q$, $\Upsilon \not\subseteq \Phi$.

Consider the cluster $\hat{\Phi} = \Phi \cup \Omega_m^{(1)}$. We claim that for $S \neq \varepsilon$, $\hat{\Phi} / S = \Phi / S$ or $\hat{\Phi} / S = \Omega_m^{(0)} = \{\varepsilon\}$. Consider the chain of equivalences:
\[
\begin{split}
X \in \hat{\Phi} / S
&\Longleftrightarrow
X \uplus S \in \hat{\Phi} = \Phi \cup \Omega_m^{(1)} \\
&\Longleftrightarrow
X \uplus S \in \Phi \,\vee\, X \uplus S \in \Omega_m^{(1)} \\
&\Longleftrightarrow
X \in \Psi / S \,\vee\, X \uplus S \in \Omega_m^{(1)}.
\end{split}
\]
Under the assumpion that $S \neq \varepsilon$, the condition $X \uplus S \in \Omega_m^{(1)}$ is equivalent to $X = \varepsilon$ and $\card{S} = 1$. The claim thus follows.

Since $\Phi$ is a minimal nonmember of $\cl{T}$ with respect to quotients and $\Omega_m^{(0)} \in \cl{T}$, by the above claim we have that $\hat{\Phi} / S \in \cl{T}$ whenever $\card{S} \geq 1$. Since $\cl{T}$ is closed under dividends, we have that $\hat{\Phi} \in \cl{T}$, and hence $\Upsilon \subseteq \hat{\Phi}$. Thus, there exists an $m$-tuple $\vect{s} \in A^m$ such that $\{\vect{s}\} \in \Upsilon \setminus \Phi$.

Let $\mtrx{M} = (\vect{m}^1, \dotsc, \vect{m}^n)$ be a $\mu \times n$ matrix whose first $m$ rows are the rows of $\mtrx{D}$ (i.e., $\bigl( \vect{m}^1(i), \dotsc, \vect{m}^n(i) \bigr) = \bigl( \vect{d}^1(i), \dotsc, \vect{d}^n(i) \bigr)$ for every $i \in m$) and whose other rows are the remaining distinct $n$-tuples in $A^n$; every $n$-tuple in $A^n$ is a row of $\mtrx{M}$ and there is no repetition of rows in $\mtrx{M}$. Note that $m \leq \mu$ and $\mu$ is infinite if and only if $A$ is infinite.

Let $\Theta = \bigcap \{\Phi' \in \cl{T}^\infty : \mtrx{M} \prec \Phi'\}$. There must exist a $\mu$-tuple $\vect{u} = (u_t \mid t \in \mu)$ in $A^\mu$ such that $\vect{u}(i) = \vect{s}(i)$ for all $i \in m$ and $\{\vect{u}\} \in \Theta$. For, if this is not the case, then the projection of $\Theta$ to its first $m$ coordinates would be a member of $\cl{T}$ containing $Q$ but not containing $\{\vect{s}\}$, contradicting the choice of $\vect{s}$.

We can now define a function $g \colon A^n \to A$ by the rule $g \mtrx{M} = \vect{u}$. The definition is valid, because every $n$-tuple in $A^n$ occurs exactly once as a row of $\mtrx{M}$. It is clear that $g$ does not satisfy $\Phi$, because $\mtrx{D} \prec \Phi$ but $g \mtrx{D} = \vect{s} \nprec \Phi$.

We need to show that every cluster in $\cl{T}$ is satisfied by $g$. Suppose, on the contrary, that there is a $\rho$-ary cluster $\Phi_0 \in \cl{T}$ which is not satisfied by $g$. Thus, for some $\rho \times n'$ matrix $\mtrx{N} = (\vect{c}^1, \dotsc, \vect{c}^{n'}) \prec \Phi_0$, with $\mtrx{N}_0 = (\vect{c}^1, \dotsc, \vect{c}^n)$, $\mtrx{N}_1 = (\vect{c}^{n+1}, \dotsc, \vect{c}^{n'})$, we have $[g \mtrx{N}_0 | \mtrx{N}_1] \nprec \Phi_0$. Let $\Phi_1 = \Phi_0 / \mtrx{N}_1^*$. Since $\cl{T}$ is closed under quotients, $\Phi_1 \in \cl{T}$. We have that $\mtrx{N}_0 \prec \Phi_1$ but $g \mtrx{N}_0 \nprec \Phi_1$, so $g$ does not satisfy $\Phi_1$ either. Define $h \colon \rho \to \mu$ to be any map such that
\[
\bigl( \vect{c}^1(i), \dotsc, \vect{c}^n(i) \bigr) = \bigl( (\vect{m}^1 h)(i), \dotsc, (\vect{m}^n h)(i) \bigr)
\]
for every $i \in \rho$, i.e., row $i$ of $\mtrx{N}_0$ is the same as row $h(i)$ of $\mtrx{M}$, for each $i \in \rho$. Let $\Phi_h$ be the $\mu$-ary simple $\infty$-minor of $\Phi_1$ via $H = \{h\}$. Note that $\Phi_h \in \cl{T}^\infty$.

We claim that $\mtrx{M} \prec \Phi_h$. To prove this, it is enough to show that $(\vect{m}^1 h, \dotsc, \vect{m}^n h) \linebreak[0] \prec \Phi_1$. In fact, we have for $1 \leq j \leq n$,
\[
\vect{m}^j h = (\vect{m}^j h(i) \mid i \in \rho) = (\vect{c}^j (i) \mid i \in \rho) = \vect{c}^j,
\]
and $(\vect{c}^1, \dotsc, \vect{c}^n) = \mtrx{N}_0 \prec \Phi_1$.

Next we claim that $\{\vect{u}\} \notin \Phi_h$. For this, it is enough to show that $\vect{u} h \nprec \Phi_1$. For every $i \in \rho$, we have
\[
\begin{split}
(\vect{u} h)(i)
&= \bigl( g(\vect{m}^1, \dotsc, \vect{m}^n) h \bigr) (i) \\
&= g \bigl( (\vect{m}^1 h)(i), \dotsc, (\vect{m}^n h)(i) \bigr)
= g \bigl( \vect{c}^1(i), \dotsc, \vect{c}^n(i) \bigr).
\end{split}
\]
Thus $\vect{u} h = g \mtrx{N}_0$. Since $g \mtrx{N}_0 \nprec \Phi_1$, we conclude that $\{\vect{u}\} \notin \Phi_h$.

Thus, $\Phi_h$ is a cluster in $\cl{T}^\infty$ that contains $\mtrx{M}$ but does not contain $\{\vect{u}\}$. By the choice of $\vect{u}$, this is impossible. We conclude that $g$ satisfies every cluster in $\cl{T}$. This completes the proof of the theorem.
\end{proof}

\section*{Acknowledgements}

The author would like to thank Ivo Rosenberg for suggesting the problem of finding a characterization of the classes of operations that are closed under Mal'cev's iterative algebra operations except for identification of variables. The author is thankful to Miguel Couceiro, Stephan Foldes, and Ross Willard for useful discussions of the topic and constructive remarks and suggestions concerning an early version of this manuscript.

\end{document}